\documentclass[10pt,reqno,oneside,a4paper]{amsart}

\usepackage[english]{babel}
\usepackage[utf8]{inputenc}
\usepackage{amsmath}
\usepackage{amsthm}
\usepackage{amssymb}
\usepackage{amsfonts}
\usepackage{graphicx}
\usepackage[shortlabels]{enumitem}
\usepackage{geometry}
\usepackage{adjustbox}
\usepackage{listings}
\usepackage{color}
\usepackage{upgreek}
\usepackage{mathtools}
\usepackage{color}
\usepackage{empheq}
\usepackage{bm}
\usepackage{csquotes}
\usepackage{tikz}
\usetikzlibrary{matrix}
\usepackage{tikz-cd}
\usepackage{setspace}
\usepackage{commath}
\usepackage[norelsize]{algorithm2e}
\usepackage{adjustbox}
\usepackage{mathrsfs}
\usepackage{stmaryrd}
\usetikzlibrary{decorations.pathmorphing,shapes}
\usepackage{rotating}
\usepackage{pdflscape}
\usepackage{subcaption}
\usepackage{afterpage}
\usepackage[all,cmtip]{xy}
\usepackage{tabu}
\usepackage{stackengine}
\usepackage{caption}
\usetikzlibrary{arrows,calc}
\usepackage{hyperref}

\newcommand{\tikzAngleOfLine}{\tikz@AngleOfLine}
\def\tikz@AngleOfLine(#1)(#2)#3{%
\pgfmathanglebetweenpoints{%
\pgfpointanchor{#1}{center}}{%
\pgfpointanchor{#2}{center}}
\pgfmathsetmacro{#3}{\pgfmathresult}%
}

\newcommand{\gap}{\hspace{1pt}}

\newcommand{\NN}{\mathbb{N}}

\renewcommand{\mod}{\text{\normalfont mod-}}

\DeclareMathOperator{\GKdim}{GKdim}

\DeclareMathOperator{\idim}{i.\hspace{-1pt}dim}
\DeclareMathOperator{\clKdim}{cl.\hspace{-1pt}Kdim}
\DeclareMathOperator{\height}{ht}
\DeclareMathOperator{\End}{End}

\DeclareMathOperator{\Spec}{Spec}

\DeclareMathOperator{\MaxSpec}{MaxSpec}
\DeclareMathOperator{\Hom}{Hom}
\DeclareMathOperator{\Frac}{Frac}

\DeclareMathOperator{\Ext}{Ext}
\DeclareMathOperator{\Stab}{Stab}

\DeclareMathOperator{\hilb}{hilb}

\DeclareMathOperator{\GL}{GL}
\DeclareMathOperator{\SL}{SL}

\DeclareMathOperator{\lcm}{lcm}
\DeclareMathOperator{\PIdeg}{PIdeg}
\DeclareMathOperator{\rank}{rank}
\DeclareMathOperator{\finsubgroup}{\hspace{1pt}\stackunder[0pt]{$\leqslant$}{$\scriptscriptstyle{\text{\normalfont{fin}}}$}\hspace{1pt}}

\newcommand{\hash}{\hspace{1pt} \# \hspace{1pt}}

\numberwithin{equation}{section}

\theoremstyle{definition}

\theoremstyle{plain}
\newtheorem{thm}[equation]{Theorem}
\newtheorem{prop}[equation]{Proposition}
\newtheorem{lem}[equation]{Lemma}
\newtheorem{cor}[equation]{Corollary}

\newtheorem*{thm*}{Theorem}

\theoremstyle{remark}
\newtheorem{rem}[equation]{Remark}

\newtheorem{example}[equation]{Example}

\definecolor{mygray}{gray}{0.8}

\newcommand\restr[2]{{
  \left.\kern-\nulldelimiterspace 
  #1 
  \right|_{#2} 
  }}

\newcounter{sarrow}

\newcounter{darrow}

\makeatletter
\def\maketag@@@#1{\hbox{\m@th\normalfont\normalsize#1}}
\makeatother

\title[Azumaya Skew Group Algebras]{Azumaya Skew Group Algebras and an Application to Quantum Kleinian Singularities}
\author{Simon Crawford}
\address{School of Mathematics, University of Edinburgh, Edinburgh, EH9 3FD, UK}
\email{S.P.Crawford@sms.ed.ac.uk}
\date{\today}

\begin{document}

\begin{abstract}
We provide easily-verified necessary and sufficient conditions for a skew group ring, or more generally, a crossed product ring, to be an Azumaya algebra. We use our results to show that (suitable localisations of) skew group rings associated to the quantum Kleinian singularities introduced in \cite{ckwz} are Azumaya, and use this to show that these algebras are maximal orders. We also give a new proof of Auslander's Theorem for quantum Kleinian singularities.
\end{abstract}

\maketitle

\section{Introduction}
Throughout, let $\Bbbk$ be an algebraically closed field of characteristic zero. In spite of the vast literature on Azumaya algebras and skew group algebras, given a specific Azumaya $\Bbbk$-algebra $A$ and an action by a finite group $G$, there is no known simple geometric criterion for the action of $G$ on $A$ which ensures that $A \hash G$ is Azumaya. While necessary and sufficient conditions for a skew group ring, or more generally a crossed product ring, to be Azumaya are given in \cite{alfaro, carvalho, paques}, these are often quite difficult to apply even in relatively basic examples. Additionally, some of the existing results in the literature require ambient hypotheses on both $A$ and the action of $G$ which are frequently not satisfied for specific skew group algebras which may be of interest. In particular, these hypotheses do not hold for a number of the \emph{quantum Kleinian singularities} (defined in Section \ref{preliminaries}) that were the original motivation for our work. One of the main purposes of this article is to provide easily-verified necessary and sufficient conditions for a skew group ring or crossed product ring to be Azumaya. Our first result considers crossed product rings where the action is $X$-outer:

\begin{thm}[Theorem \ref{outerazumaya}] \label{outerintro}
Consider a crossed product $T \coloneqq A*G$, where $A$ is a prime noetherian $\Bbbk$-algebra and where $G$ is a finite group. Suppose that the action of $G$ on $A$ is $X$-outer, and that $G$ acts as a group of $\Bbbk$-linear automorphisms on $Z(A)$. Then $T$ is prime noetherian. Moreover, $T$ is Azumaya if and only if
\begin{enumerate}[{\normalfont (1)},leftmargin=30pt,topsep=0pt,itemsep=0pt]
\item $A$ is Azumaya; and
\item $G$ acts freely on $Z(A)$; that is, the stabiliser of every maximal ideal of $Z(A)$ is trivial.
\end{enumerate}
If $T$ is Azumaya, then the rank of $T$ over its centre satisfies $\rank T = |G|^2 \rank A$.
\end{thm}

\indent If $A$ is commutative and $T$ is a skew group ring, this can be stated more concisely as follows:

\begin{cor} \label{corollaryintro}
Consider a skew group ring $T \coloneqq A \hash G$, where $A$ is a commutative noetherian $\Bbbk$-algebra which is a domain and where $G$ is a finite group. Then $T$ is prime noetherian. Moreover, $T$ is Azumaya if and only if $G$ acts freely on $A$, and in this case, $\rank T = |G|^2$.
\end{cor}

In our second main result, we consider the case of a skew group ring $A \hash G$ where a cyclic group $G$ acts by inner automorphisms:

\begin{thm}[Theorem \ref{innerazumaya}] \label{innerintro}
Let $A$ be a prime $\Bbbk$-algebra and let $G$ be a cyclic group acting inner on $A$. Suppose also that $T = A \hash G$ is prime. Then $A$ is Azumaya if and only if $T$ is Azumaya, and in this case they have the same rank over their centres.
\end{thm}

Our main application of these results is to \emph{quantum Kleinian singularities}, which were introduced in \cite{ckwz} and which we now briefly describe. \\
\indent In \cite{ckwz14}, the authors classified all pairs $(A,G)$ where $A$ is an AS-regular algebra of global dimension 2, and $G$ is a finite group acting faithfully on $A$ and with trivial homological determinant. These conditions ensure that these actions can be thought of as noncommutative analogues of the action of a finite subgroup of $\SL(2,\Bbbk)$ on $\Bbbk[u,v]$. Table 1 in Section \ref{preliminaries} lists all of the possibilities for the pairs $(A,G)$. Following \cite{ckwz}, we refer to the invariant rings $A^G$ as \emph{quantum Kleinian singularities}. It is natural to ask to what extent the results of the classical McKay correspondence for Kleinian singularities hold in the quantum setting. One such classical result is Auslander's Theorem, which says that if $G$ is a finite subgroup of $\SL(2,\Bbbk)$ acting on $R = \Bbbk[u,v]$, then there is a natural graded isomorphism $\End_{R^G}(R) \cong R \hash G$. An analogue of Auslander's Theorem was shown to hold for quantum Kleinian singularities in \cite{ckwz}, and quantum analogues of other results in the classical McKay correspondence have been established in \cite{ckwz2}. \\
\indent Our original motivation was to generalise other results for Kleinian singularities to the quantum setting, namely those found in \cite{cbh}.  A relatively innocuous result in [loc. cit.] is that (a deformation of) the skew group ring $\Bbbk[u,v] \hash G$, where $G$ is a finite subgroup of $\SL(2,\Bbbk)$, is a maximal order. In the classical setting considered in [loc. cit.], this result follows quickly from results in the literature, but the analogous result in the quantum setting is more difficult to establish. Our approach involves showing that suitable localisations of the skew group ring $A \hash G$ are Azumaya, and this necessitated proving Theorems \ref{outerintro} and \ref{innerintro}. Using this fact, we prove the following:

\begin{thm}[Theorem \ref{maxorderthm}, Corollary \ref{invmaxorder}] \label{maxintro}
Suppose that $A^G$ is a quantum Kleinian singularity. Then $A \hash G$ and $A^G$ are maximal orders.
\end{thm}

Finally, as a corollary we give a new proof of Auslander's Theorem for quantum Kleinian singularities, which was established in \cite{ckwz} using very different methods. We remark that our approach also allows one to establish a similar result for deformations of $A \hash G$ and $A^G$, which is not the case using the techniques of \cite{ckwz}.

\begin{thm}[Theorem \ref{auslandercor}] \label{auslanderintro}
Suppose that $A^G$ is a quantum Kleinian singularity. Then
\begin{align*}
\End_{A^G}(A) \cong A \hash G.
\end{align*}
\end{thm}

\subsection{Organisation of this paper} In Section 2, we recall some basic definitions and facts, and we also fix our notation. In Sections \ref{outersec} and \ref{innersec}, we prove Theorems \ref{outerintro} and \ref{innerintro}, and provide some examples demonstrating the necessity of our hypotheses. Section 5 shows that, for particular quantum Kleinain singularities $A^G$, suitable localisations of the skew group rings $A \hash G$ are Azumaya. Finally, in Sections 6 and 7 we prove Theorems \ref{maxintro} and \ref{auslanderintro}.

\subsection{Acknowledgements} The author is an EPSRC-funded student at the University of Edinburgh, and material contained in this paper will form part of his PhD thesis. The author would like to thank his supervisor Susan J. Sierra for her guidance, James Zhang, Chelsea Walton and Ken Brown for helpful discussions, and the EPSRC.

\section{Preliminaries} \label{preliminaries}
\indent We now recall some definitions and results that will be used throughout this paper, and we provide the classification of quantum Kleinian singularities outlined in the introduction. In this section, $R$ will denote an arbitrary ring.

\subsection{Definitions, notation, and basic results}
We fix an algebraically closed field $\Bbbk$ of characteristic $0$ throughout this paper. We write $\mod R$ (respectively, $R\text{-mod}$) for the category of finitely generated right (respectively, left) $R$-modules; we shall work with right $R$-modules unless stated otherwise. When we speak of a noetherian ring, we mean a ring which is both right and left noetherian. \\
\indent Let $A$ be a ring with centre $Z$. There is a natural map
\begin{align*}
\theta : A \otimes_Z A^{\text{op}} \to \End_Z(A), \qquad a \otimes b \mapsto (r \mapsto arb).
\end{align*}
We say that $A$ is an \emph{Azumaya algebra (over its centre $Z$)} if $A$ is a finitely generated projective $Z$-module and $\theta : A \otimes_Z A^{\text{op}} \to \End_Z(A)$ is an isomorphism. \\
\indent Closely related is the notion of a \emph{separable extension}. Given an extension of rings $R \subseteq S$, $S$ is said to be a separable extension of $R$ if there exists an element $\sum_{i=1}^m s_i \otimes s_i'$ of $S \otimes_R S$ such that $\sum_{i=1}^m s_i s_i' = 1$ and $\sum_{i=1}^m x s_i \otimes s_i' = \sum_{i=1}^m s_i \otimes s_i' x$ for all $x \in S$. We can then characterise an Azumaya algebra as being a ring which is separable over its centre. \\
\indent Clearly any commutative ring is Azumaya, and the same is true of any matrix ring over a commutative ring \cite[Proposition 13.7.7]{mandr}. For another example, let $q \in \Bbbk^\times$ and consider the quantum torus $\Bbbk_q[u^{\pm 1},v^{\pm 1}]$, which is the algebra with generators $u^{\pm 1}$ and $v^{\pm 1}$ subject to the relation $vu = quv$. If $q$ is a root of unity, then this algebra is Azumaya \cite[Example III.1.4]{browngoodearl}. We will frequently make use of the fact that the localisation of an Azumaya algebra at an Ore set consisting of central elements is again Azumaya. \\
\indent By \cite[III.1.4]{browngoodearl}, if $A$ is a prime Azumaya algebra, then $n^2 \coloneqq \dim_{Z/\mathfrak{m}} (A/\mathfrak{m}A)$ is constant as $\mathfrak{m}$ varies over maximal ideals of $Z$, and this value is necessarily square. In this case, we say that $A$ has \emph{rank} $n^2$ (over its centre). \\
\indent An important result in Azumaya theory is the Artin-Procesi Theorem, which characterises prime Azumaya algebras in terms of polynomial identities, and in particular in terms of the PI (polynomial identity) degree of factors of the algebra. Any undefined terms in what follows can be found in \cite[Appendix I.13]{browngoodearl}.

\begin{thm}[Artin-Procesi]
Suppose that $R$ is a prime ring with centre $Z$. Then the following are equivalent:
\begin{enumerate}[{\normalfont (1)},leftmargin=30pt,topsep=0pt,itemsep=0pt]
\item $R$ is an Azumaya algebra of rank $d^2$;
\item $R$ is a PI ring of PI degree $d$ and $\PIdeg R/P = d$ for all $P \in \Spec R$; and
\item $R$ is a PI ring of PI degree $d$ and $\PIdeg R/M = d$ for all $M \in \MaxSpec R$.
\end{enumerate}
\end{thm}

By combining the Artin-Procesi Theorem with results in the literature, one obtains the following characterisation of an Azumaya algebra which we will use throughout this paper.

\begin{lem} \label{azumayalem}
Suppose that $R$ is a prime $\Bbbk$-algebra which is finite over its centre $Z$ (hence PI) and of PI degree $d$. Then $R$ is an Azumaya algebra of rank $d^2$ if and only if $R/\mathfrak{m}R \cong M_d(\Bbbk)$ for all $\mathfrak{m} \in \MaxSpec Z$.
\end{lem}
\begin{proof}
($\Rightarrow$) First suppose that $R$ is Azumaya of rank $d^2$ and consider $\mathfrak{m} \in \MaxSpec Z$. By \cite[Proposition 13.7.9]{mandr}, $\mathfrak{m}R$ is a maximal ideal of $R$, and so $\PIdeg R/\mathfrak{m}R = d$ by the Artin-Procesi Theorem. By \cite[Theorem III.1.6]{browngoodearl}, it follows that $R/\mathfrak{m}R \cong M_d(\Bbbk)$. \\
\indent ($\Leftarrow$) Now let $M$ be a maximal ideal of $R$, and set $\mathfrak{m} = M \cap Z$, which is a maximal ideal of $Z$. Then $M/\mathfrak{m}R$ is a proper ideal of the ring $R/\mathfrak{m}R$, which is isomorphic to $M_d(\Bbbk)$ by hypothesis. Since this is a simple ring we must have $M/\mathfrak{m}R = 0$, and so $M = \mathfrak{m}R$. In particular, $\PIdeg R/M = d$, and so by the Artin-Procesi Theorem, $R$ is an Azumaya algebra of rank $d^2$.
\end{proof}

\indent Let $G$ be a group acting (on the left) on a ring $R$. The \emph{skew group ring} $R \hash G$ is then the free left $R$-module with the elements of $G$ as a basis, with multiplication extended linearly from the rule $(rg)(sh) = r(g \cdot s) \gap gh$ for $r,s \in R$, $g,h, \in G$, where $g \cdot s$ is the image of $s$ under the action of $g$. \\
\indent The notion of a crossed product is a generalisation of the skew group ring construction. Suppose that $R$ is a ring and $G$ is a group. Then a \emph{crossed product} $R*G$ of $R$ by $G$ is a ring containing a copy of $R$ and a set of units $\overline{G} = \{\overline{g} \mid g \in G \}$ which is in bijection with $G$, such that:
\begin{enumerate}[{\normalfont (1)},topsep=0pt,itemsep=0pt,leftmargin=30pt]
\item $R*G$ is a free left $R$-module with basis $\overline{G}$, and where $\overline{e} = 1$;
\item If $g \in G$ then $\overline{g}R = R \overline{g}$; and 
\item If $g,h \in G$ then $R \overline{g}\overline{h} = R \overline{gh}$.
\end{enumerate}
The second condition implies that each element $\overline{g}$ induces an automorphism $\alpha_g$ of $R$ via $\overline{g}r = \alpha_g(r)\overline{g}$; however, in general, the set $\{ \alpha_g \mid g \in G \}$ is not a group. Despite this, it will still be convenient to say that $G$ acts on $R$, and to use standard group-theoretic terminology. The third condition implies that there exists a map $\tau : G \times G \to U(R)$ such that $\overline{g}\overline{h} = \tau(g,h) \overline{gh}$. A skew group ring is the special case where $\tau(g,h) = 1$ for all $g,h \in G$. \\
\indent Later we will wish to restrict attention to so-called $X$-outer automorphisms, and to define this we need to recall another definition. Let $R$ be a prime ring. The \emph{symmetric Martindale ring of quotients} of $R$ is the ring extension $Q_s(R)$ uniquely determined by the following properties:

\begin{enumerate}[{\normalfont (1)},leftmargin=30pt,topsep=0pt,itemsep=0pt]
\item If $q \in Q_s(R)$ then there exist nonzero ideals $I,J$ of $R$ such that $Iq, qJ \subseteq R$;
\item Let $q \in Q_s(R)$ and let $I$ be a nonzero ideal of $R$. If either $Iq = 0$ or $qI = 0$ then $q = 0$;
\item Let $I,J$ be nonzero ideals of $R$ and let $f : {}_R I \to {}_R R$ and $g : J_R \to R_R$ be module homomorphisms satisfying $f(a)b = a g(b)$ for all $a \in I$ and $b \in J$. Then there exists $q \in Q_s(R)$ such that $f(a) = aq$ and $g(b) = qb$ for all $a \in I$ and $b \in J$.
\end{enumerate}

\indent We remark that, in the case of a prime PI ring, $Q_s(R)$ is obtained by inverting the nonzero central elements of $R$, see \cite[Theorem 23.4]{passman}. \\
\indent We can now define what we mean by $X$-inner and $X$-outer automorphisms. An automorphism of a prime ring $R$ is said to be \emph{$X$-inner} if it is equal to conjugation by an element of $Q_s(R)$, and it is called \emph{$X$-outer} otherwise. If a group $G$ acts on a prime ring $R$, we write $G_{X\text{-inn}}$ for the subgroup of $G$ consisting of elements which act as $X$-inner automorphisms. The action is said to be \emph{$X$-outer} if $G_{X\text{-inn}}$ is trivial. \\
\indent Given a semiprime noetherian ring $R$, write $Q(R)$ for its Goldie quotient ring. Then $R$ is said to be a \emph{maximal order} if there exists no order $R'$ with $R \subset R' \subseteq Q(R)$ and with $aR'b \subseteq R$ for some nonzero $a,b \in R$. This should be viewed as a noncommutative analogue of a noetherian domain being integrally closed, see \cite[Proposition 5.1.3]{mandr}. \\
\indent We have the following equivalent characterisations of the property of being a maximal order. One of these characterisations appears in \cite[Lemma 2.1]{martin} but without proof, so we provide one. We recall that for a nonzero ideal $I$ of $R$, we write
\begin{align*}
O_\ell(I) \coloneqq \{ q \in Q(R) \mid qI \subseteq I \}, \qquad O_r(I) \coloneqq \{ q \in Q(R) \mid Iq \subseteq I \}.
\end{align*}

\begin{lem}[Martin] \label{maxorderlem}
Let $R$ be a prime Noetherian ring. Then the following are equivalent:
\begin{enumerate}[{\normalfont (1)},leftmargin=*,topsep=0pt,itemsep=0pt]
\item $R$ is a maximal order;
\item $\End_R(I) = R$ for all nonzero ideals $I$ of $R$; and
\item $\End_R(\mathfrak{p}) = R$ for all nonzero prime ideals $\mathfrak{p}$ of $R$.
\end{enumerate}
\end{lem}
\begin{proof}
Throughout write $Q = Q(R)$. The equivalence of (1) and (2) is well-known and true under weaker hypotheses, see \cite[Proposition 5.1.4]{mandr}. It is clear that (2) implies (3), so it remains to show the reverse implication. \\
\indent Suppose that (2) does not hold. By the noetherian hypothesis, choose an ideal $I$ maximal among those ideals satisfying $R \subsetneq \End_R(I) \subseteq Q$. We claim that $I$ is prime. Seeking a contradiction, suppose this is not the case, so there exist ideals $J,K \nsubseteq I$ with $JK \subseteq I$; without loss of generality, we may assume that $J,K \supsetneq I$. Set $H = \{r \in R \mid rK \subseteq I \}$, so that $H$ is an ideal of $R$ with $H \supseteq J \supsetneq I$. Note that if $h \in H$ and $q \in \End_R(I)$ then
\begin{align}
qhK \subseteq qI \subseteq I \subsetneq K, \label{martinlemeqn}
\end{align}
so that $qh \in \End_R(K)$. By the maximality hypothesis on $I$, we have $\End_R(K) = R$, so $qh \in R$. But (\ref{martinlemeqn}) also shows that $qhK \subseteq I$, so that $qh \in H$ by the  definition of $H$ and hence $q \in \End_R(H)$. Therefore $H$ is an ideal of $R$ with $H \supsetneq I$ and $R \subset \End_R(I) \subseteq \End_R(H) \subseteq Q$, contradicting our choice of $I$, and so $I$ must be prime. Thus conditions (2) and (3) are equivalent.
\end{proof}

Finally, we define some ring-theoretic and homological notions. A ring $R$ is said to be \emph{Gorenstein} if it is noetherian and both $\idim R_R$ and $\idim {}_R R$ are finite. By \cite[Lemma A]{zaks}, our hypotheses imply that these two values coincide, and we call this common value the \emph{(injective) dimension} of $R$. For an $R$-module $M$, the \emph{grade} of $M$ is defined to be $j(M) = \inf \{ i \mid \Ext_R^i (M,R) \neq 0\} \in \NN \cup \{ \infty \}$. If $R$ is a Gorenstein ring of finite GK dimension, then it is said to be \emph{(GK-)Cohen-Macaulay} (CM) if $\GKdim M + j(M) = \GKdim R$ for all finitely generated left and right modules $M$.

\subsection{Quantum Kleinian singularities} 
In \cite{ckwz14}, the authors classified all pairs $(A,G)$, where $A$ is an AS-regular algebra of global dimension 2 generated in degree 1, and $G$ is a finite group acting faithfully on $A$ and with \emph{trivial homological determinant}. (In fact, the authors proved a slightly stronger result concerning Hopf actions, but in the present context we focus only on the group case.) While we do not give a precise definition of this last term, we remark that if $G$ is a finite subgroup of $\GL(n,\mathbb{C})$ acting on $\mathbb{C}[x_1, \dots , x_n]$ with trivial homological determinant, then $G$ is in fact a finite subgroup of $\SL(n,\mathbb{C})$; see \cite[pp. 6]{ckwz}. The classification is as follows:

\begin{table}[h]
\begin{center}
{\renewcommand{\arraystretch}{1.3}%
\begin{tabular}{ c | c | c } 
Case & $A$ & $G$\\
\hline
(0) & \hfill $\Bbbk[u,v]$ & $ G \finsubgroup \SL(2,\mathbb{\Bbbk})$ \\ 
(i) & \hfill $\Bbbk_q[u,v]$ & $C_{n+1}$ \\ 
(ii) & \hfill $\Bbbk_{-1}[u,v]$ & $ S_2$ \\ 
(iii) & \hfill $\Bbbk_{-1}[u,v]$ & $D_{n}$ \\ 
(iv) & \hfill $\Bbbk_J[u,v]$ & $C_2$\\ 
\end{tabular}}
\\[10pt]
\caption*{Table 1: The pairs $(A,G)$.}
\end{center}
\end{table}
\vspace{-10pt}
\noindent Observe that this classification includes the case of finite subgroups of $\SL(2,\Bbbk)$ acting on a polynomial ring in two variables. We will often refer to case (0) as the classical case, and the remaining cases as the quantum cases. For case (0), we refer to the rings $A^G$ as classical Kleinian singularities, and following \cite{ckwz}, for the remaining cases we call the rings $A^G$ \emph{quantum Kleinian singularities}. Henceforth, when we say that $A^G$ is a quantum Kleinian singularity, we mean that the pair $(A,G)$ is a pair from cases (i)-(iv) of Table 1. We will not consider case $(0)$ for the remainder of this paper, other than to draw comparisons between it and the quantum cases. \\
\indent We briefly explain the notation and how the groups act. The algebras $\Bbbk_q[u,v]$ (where $q \in \Bbbk^\times$) and $\Bbbk_J[u,v]$ are, respectively, the \emph{quantum plane} and \emph{Jordan plane}, and have presentations
\begin{align*}
\Bbbk_q[u,v] = \frac{\Bbbk \langle u,v \rangle}{\langle vu - quv \rangle} \quad \text{and} \quad \Bbbk_J[u,v] = \frac{\Bbbk \langle u,v \rangle}{\langle vu - uv - u^2 \rangle}.
\end{align*}
These algebras are both noetherian domains of global dimension $2$, and may be thought of as noncommutative analogues of $\Bbbk[u,v]$. \\
\indent The groups $C_n, S_2$, and $D_n$ are, respectively, the cyclic group of order $n$, the symmetric group on two letters, and the dihedral group of order $2n$ (obviously $S_2 \cong C_2$, but our choice of notation will become clear soon). We will frequently make use of the abstract presentations
\begin{align*}
C_n = \langle g \mid g^n \rangle, \quad S_2 = \langle h \mid h^2 \rangle \quad \text{and} \quad D_n = \langle g,h \mid g^n, h^2, (hg)^2 \rangle.
\end{align*}
Throughout the remainder of this paper, given a quantum Kleinian singularity $A^G$, $g$ will denote an element of order $n$ (where $n$ is to determined from context) and $h$ will denote an element of order $2$. These elements will act on $u,v \in A$ via
\begin{align*}
g \cdot u = \omega u, \quad g \cdot v = \omega^{-1} v, \quad h \cdot u = v, \quad h \cdot v = u,
\end{align*}
where $\omega$ is a primitive $n$th root of unity. One can verify that these give rise to well-defined actions of the groups from Table 1 on the corresponding algebras.


\section{Azumaya crossed products with an $X$-outer action} \label{outersec}

In this section, we prove Theorem \ref{outerintro} from the introduction, and provide an example which demonstrates an approach we will use in later sections. Suppose that $A*G$ is a crossed product where $G$ is a finite group acting $X$-outer on $A$. Unsurprisingly, a necessary condition to ensure that $A*G$ is Azumaya is that $A$ is itself Azumaya. Additionally, since properties of an Azumaya algebra $A$ are closely related to the geometry of $\Spec Z(A)$, one should expect the action of $G$ on the points of $\Spec Z(A)$ to influence whether $A*G$ is Azumaya. This is the case, as we will see in Theorem \ref{outerazumaya}. \\
\indent We first determine the centre $Z(A * G)$ of $A*G$, which is relatively straightforward.

\begin{lem} \label{outercentre}
Consider a crossed product $T \coloneqq A * G$, where $A$ is prime and $G$ is a finite group. Suppose that the action of $G$ on $A$ is $X$-outer. 
Then $Z(A * G) = Z(A)^G$.
\end{lem}
\begin{proof}
The inclusion $Z(A)^G \subseteq Z(A * G)$ is easy to establish and holds without the $X$-outer hypothesis, so we now show the reverse inclusion. So suppose that $\sum_{g \in G} c_g g \in Z(A * G)$. Then, for any $b \in A$ we have
\begin{align*}
\sum_{g \in G} b c_g \overline{g} = \sum_{g \in G} c_g \overline{g} b = \sum_{g \in G} c_g \gap \alpha_g(b) \overline{g}.
\end{align*}
This forces $c_g \alpha_g ( b) = b c_g$ for each $g \in G$ and each $b \in A$. In particular $c_g$ is a normal element of $A$, and so if it were nonzero then it would be an element of $Q_s(A)$ by \cite[Lemma 2.1 (ii)]{passman2}. Since $\alpha_g$ acts $X$-outer when $g \neq e$ by hypothesis, this forces $c_g = 0$ for all $g \neq e$. Therefore we must have $Z(A * G) \subseteq Z(A)$, so now consider some $a \in Z(A * G) \subseteq Z(A)$. Then, for any $\overline{g}$ we have
\begin{align*}
a \overline{g} = \overline{g} a = \alpha_g(a) \overline{g},
\end{align*}
which means that $a$ lies in $Z(A)^G$. This gives the claimed equality.
\end{proof}

We also need the following lemma:

\begin{lem} \label{Gmaxideal}
Let $R$ be a commutative finitely generated $\Bbbk$-algebra which is a domain. Suppose that $G$ is a finite group acting $\Bbbk$-linearly and freely on $R$; that is, the stabiliser of every maximal ideal of $R$ is trivial. If $\mathfrak{m} \in \MaxSpec R^G$, then $\mathfrak{m} R$ is a $G$-maximal ideal of $R$.
\end{lem}
\begin{proof}
Let $\mathfrak{m} \in \MaxSpec R^G$. It is straightforward to check that $\mathfrak{m}R$ is a $G$-stable ideal of $R$, so it remains to verify that it is maximal among all such ideals. Since $\Bbbk$ is algebraically closed, there is a one-to-one correspondence between orbits of points in $\Spec R$ and $G$-maximal ideals of $R$, and by our hypothesis on maximal ideals, these orbits have size $|G|$. Since $\Spec R/\mathfrak{m}R$ is a $G$-stable subvariety of $\Spec R$, if we show that it has dimension $|G|$ over $\Bbbk$ then it follows that $\mathfrak{m}R$ is a $G$-maximal ideal.\\
\indent To this end, consider the natural morphism $\pi : \Spec R \to \Spec R^G$ coming from the inclusion $i : R^G \hookrightarrow R$. Since $G$ is finite, the action of $G$ on $\Spec R$ is closed, and so \cite[Amplification 1.3]{mumford} tells us that $(\Spec R^G, \pi)$ is a geometric quotient of $R$. By \cite[Proposition 0.9]{mumford}, it follows that $\pi$ is flat and finite. Therefore $\pi$ is a flat, finite, dominant morphism between integral schemes, and so writing $p = \Spec R^G/\mathfrak{m}$, we have $|\pi^{-1}(p) | = [\Frac(R) : \Frac(R^G)]$ by \cite[Exercise 5.1.25]{liu}, and this value is equal to $|G|$ by \cite[Proposition 1.1.1]{benson}. But 
\begin{align*}
\pi^{-1}(p) \cong p \times_{\Spec R^G} \Spec R \cong \Spec ( R^G/ \mathfrak{m} \otimes_{R^G} R) \cong \Spec R/\mathfrak{m}R,
\end{align*}
and so $|\Spec R/\mathfrak{m}R| = |G|$. By the preceding paragraph, it now follows that $\mathfrak{m}R$ is a $G$-maximal ideal of $R$.
\end{proof}


\begin{thm} \label{outerazumaya}
Consider a crossed product $T \coloneqq A*G$, where $A$ is a prime noetherian $\Bbbk$-algebra and where $G$ is a finite group. Suppose that the action of $G$ on $A$ is $X$-outer, and that $G$ acts as a group of $\Bbbk$-linear automorphisms on $Z(A)$. Then $T$ is prime noetherian. Moreover, $T$ is Azumaya if and only if
\begin{enumerate}[{\normalfont (1)},leftmargin=30pt,topsep=0pt,itemsep=0pt]
\item $A$ is Azumaya; and
\item $G$ acts freely on $Z(A)$; that is, the stabiliser of every maximal ideal of $Z(A)$ is trivial.
\end{enumerate}
If $T$ is Azumaya, then the ranks of $A$ and $T$ satisfy $\rank T = |G|^2 \rank A$.
\end{thm}
\begin{proof}
First note that $T$ is prime by \cite[Corollary 12.6]{passman} and noetherian by \cite[Lemma 1.5.1]{mandr}, and that $Z(A)$ is a domain since $A$ is prime, and is noetherian by \cite[13.7.10]{mandr} since $A$ is Azumaya and noetherian. \\ 
\indent Throughout this proof, we will be concerned with maximal ideals lying in a number of different rings: it will be our convention to write $\mathfrak{m}, \mathfrak{n}$, and $M$ for maximal ideals in $Z(T), Z(A),$ and $T$, respectively. \\
\indent $(\Rightarrow)$ Assume that $T$ is Azumaya. Since $G$ acts $X$-outer, $Z(T) = Z(A)^G \subseteq A$ by Lemma \ref{outercentre}. Therefore, \cite[Proposition 2.4]{carvalho} implies that $A$ is a separable extension of $Z(A)^G$. We then have a chain of inclusions $Z(A)^G \subseteq Z(A) \subseteq A$, and so $A$ is a separable extension of $Z(A)$ by \cite[Proposition 2.5]{hirata}; that is, $A$ is Azumaya. \\
\indent It remains to show that (2) holds, and that if $\rank T = r^2$ then $\rank A = (r/|G|)^2$. We establish these simultaneously. We claim that a maximal ideal $\mathfrak{n}$ of $Z(A)$ has trivial stabiliser if and only if $\dim_\Bbbk A/\mathfrak{n}A = (r/|G|)^2$. So suppose that $\mathfrak{n} \in \MaxSpec Z(A)$. Since $Z(A)$ is a module-finite extension of $Z(A)^G$, \cite[Lemma 1.4.2]{benson} implies that $\mathfrak{n} \cap Z(A)^G$ is a maximal ideal of $Z(A)^G = Z(T)$. Therefore, as $T$ is Azumaya, $(\mathfrak{n} \cap Z(A)^G)T$ is a maximal ideal of $T$. Writing 
\begin{align*}
\mathfrak{n}' = \bigcap_{g \in G} g \cdot \mathfrak{n},
\end{align*}
which is the intersection of $|G|/|\Stab_G(\mathfrak{m})|$ maximal ideals of $A$, we have $\mathfrak{n} \cap Z(A)^G \subseteq \mathfrak{n}'$. Maximality of $(\mathfrak{n} \cap Z(A)^G)T$ forces $(\mathfrak{n} \cap Z(A)^G)T = \mathfrak{n}' T$, and then intersecting down with $A$ shows that
\begin{align*}
(\mathfrak{n} \cap Z(A)^G)A = \mathfrak{n}'.
\end{align*}
Since $\mathfrak{n}'$ is $G$-stable and $T$ is Azumaya of rank $r^2$, we have
\begin{align*}
\dim_\Bbbk \frac{A}{\mathfrak{n}'} * G = \dim_\Bbbk \frac{T}{\mathfrak{n}' T} = r^2,
\end{align*} 
and so $\dim_\Bbbk A/\mathfrak{n'} A = r^2/ |G|$. Now, since $A$ is Azumaya, there is some $\ell \in \mathbb{N}$ such that $A/(g \cdot \mathfrak{n}) A \cong M_{\ell}(\Bbbk)$ for each $g \in G$. Then, since $\mathfrak{n'} A$ is the intersection of $|G|/|\Stab_G(\mathfrak{n})|$ maximal ideals of $A$, the Chinese Remainder Theorem implies that
\begin{align*}
\frac{A}{\mathfrak{n}'} \cong \underbrace{M_{\ell}(\Bbbk) \times \dots \times M_{\ell}(\Bbbk)}_{|G|/|\Stab_G(\mathfrak{n})| \text{ copies}},
\end{align*}
so that, taking dimensions on both sides,
\begin{align*}
\frac{r^2}{|G|} = \frac{|G|}{|\Stab_G(\mathfrak{n})|} \ell^2.
\end{align*}
Therefore $\mathfrak{n}$ has trivial stabiliser if and only if $\ell^2 = r^2/|G|^2$, which happens if and only if $\dim_\Bbbk A/\mathfrak{n}A = (r/|G|)^2$, as claimed. \\
\indent It follows that if we can find a single maximal ideal of $Z(A)$ with trivial stabiliser then $A$ has rank $(r/|G|)^2$, and since $A$ is Azumaya, 
this will imply that every maximal ideal of $Z(A)$ has trivial stabiliser. \\
\indent To this end, write $X = \Spec Z(A)$, which is an irreducible affine variety. Since $G$ acts $X$-outer it acts faithfully on $Z(A)$, so if $g \in G$ is not the identity, then $\{ x \in X \mid g \cdot x = x \}$ is a proper subvariety of $X$ and so has strictly smaller dimension than $X$. But $X$ is irreducible and so cannot be a finite union of subvarieties of strictly smaller dimension, and so some point of $X$ lies outside of  $\bigcup_{g \in G} \{ x \in X \mid g \cdot x = x \}$. That is, there exists some maximal ideal of $Z(A)$ having trivial stabiliser, completing the proof of necessity. \\
\indent $(\Leftarrow)$ Now suppose that (1) and (2) both hold. Seeking to prove that $T$ is Azumaya, first note that by Lemma \ref{outercentre}, we have $Z(T) = Z(A)^G$. Let $\mathfrak{m}$ be a maximal ideal of $Z(T)$. Then, since $Z(A)$ is a noetherian domain and the stabiliser of every maximal ideal of $Z(A)$ is trivial, Lemma \ref{Gmaxideal} implies that $\mathfrak{m}Z(A)$ is a $G$-maximal ideal of $Z(A)$. By the Azumaya property of $A$, we find that $\mathfrak{m}A$ is a $G$-maximal ideal of $A$. Let $Q$ be a prime of $A$ minimal over $\mathfrak{m}A$; necessarily $Q$ is a maximal ideal of $A$, and so by hypothesis $\Stab_G Q = \{e\}$. Since 
\begin{align*}
\frac{A/\mathfrak{m}A}{Q/\mathfrak{m}A} * \Stab_G Q \cong A/Q
\end{align*}
is prime, \cite[Corollary 14.8]{passman} implies that $(A/\mathfrak{m}A) * G$ is prime, and hence $\mathfrak{m}T$ is a prime ideal of $T$. We claim that $\mathfrak{m}T$ is in fact maximal. So suppose that $I$ is a prime ideal of $T$ with $\mathfrak{m}T \subseteq I$. Intersecting down with $A$ we find that $\mathfrak{m}A \subseteq I \cap A$, where $I \cap A$ is a $G$-stable ideal of $A$, and so $G$-maximality of $\mathfrak{m}A$ forces $\mathfrak{m}A = I \cap A$. By \cite[Theorem 14.7]{passman}, such prime ideals $I$ are in one-to-one correspondence with primes $J$ of $A * \Stab_G Q = A$ with $J \cap A = Q$. This forces $J = Q$, and so there is only one prime $I$ with $\mathfrak{m}T \subseteq I$. Since $\mathfrak{m}T$ is prime, this forces $\mathfrak{m}T = I$, and so $\mathfrak{m}T$ is a maximal ideal of $T$. \\
\indent Now let $M$ be a maximal ideal of $T$, and write $\mathfrak{m} = M \cap Z(T)$. 
By \cite[Lemma III.1.5]{browngoodearl}, $\mathfrak{m}$ is a maximal ideal of $Z(T)$, and moreover it satisfies $\mathfrak{m} T \subseteq M$. By the previous paragraph, $\mathfrak{m}T$ is a maximal ideal of $T$, and so $\mathfrak{m}T = M$. \\
\indent Since $T$ is prime, we use the Artin-Procesi Theorem to show that $T$ is Azumaya, and it suffices to show that $T$ is PI and that every maximal ideal $M$ of $T$ satisfies $(M \cap Z(T)) T = M$, see \cite[Theorem III.1.6]{browngoodearl}. To see the first of these, note that we have a chain of inclusions
\begin{align*}
Z(T) = Z(A)^G \subseteq Z(A) \subseteq A \subseteq T
\end{align*}
where each term is module-finite over the preceding term: indeed, since $Z(A)$ is a noetherian domain we find that $Z(A)$ is finite over $Z(A)^G$ by 
\cite[Proposition 5.4]{leu}, while $A$ is finite over $Z(A)$ since $A$ is Azumaya, and $T$ is finite over $A$ by definition. Therefore $T$ is PI by \cite[Corollary 13.1.13 (iii)]{mandr}. Finally, the preceding paragraph shows that every maximal ideal $M$ of $T$ satisfies $(M \cap Z(T)) T = M$, and so $T$ is Azumaya.
\end{proof}

Note that in the special case where $A * G = A \hash G$ is a skew group ring, the condition that $G$ acts as a group on $Z(A)$ is automatically satisfied. Moreover, if $A$ is commutative then the action is automatically $X$-outer, and so we obtain the statement of Corollary \ref{corollaryintro}. \\
\indent We now give an example of an application of Theorem \ref{outerazumaya}. This demonstrates an approach which we will use in Section \ref{kleinazumaya} to analyse skew group rings coming from quantum Kleinian singularities.

\begin{example} \label{outerexample}
Let $A = \Bbbk[u,v]$, which is clearly a prime noetherian Azumaya algebra. Let $G = S_2 = \langle h \rangle$ act $X$-outer on $A$ via $h \cdot u = v$, $h \cdot v = u$. Observe that a maximal ideal $\langle u-a, v - b \rangle$ of $Z(A) = A$ has nontrivial stabiliser if and only if $a = b$; in particular, the action is not free, and so Theorem \ref{outerazumaya} tells us that $T = A \hash G$ is not Azumaya. Indeed, the maximal ideal $\mathfrak{m} = \langle u+v, uv \rangle$ of $Z(T) = \Bbbk[u+v, uv]$ does not extend to a maximal ideal of $T$ since $\mathfrak{m}T \subsetneq \langle u+v, uv, h-1 \rangle \subsetneq T$, and so $T$ is not Azumaya. More generally, given $\mathfrak{m} = \langle u+v-\alpha, uv - \beta \rangle \in \MaxSpec Z(T)$, one can show that 
\begin{align*}
T/\mathfrak{m}T \cong \left \{
\begin{array}{c l}
M_2(\Bbbk) & \text{if } \alpha^2 \neq 4\beta \\
\Pi(\mathbb{A}_2) & \text{if } \alpha^2 = 4\beta
\end{array}
\right . ,
\end{align*}
where $\Pi(\mathbb{A}_2)$ is the preprojective algebra of an $\mathbb{A}_2$ Dynkin quiver, so that $\mathfrak{m}T$ is not even prime when $\alpha^2 = 4 \beta$. However, if we replace $A$ by $A' = A[(u-v)^{-1}]$ then the maximal ideals of $Z(A') = A'$ have the form $\langle u-a, v - b \rangle$ with $a \neq b$, and Theorem \ref{outerazumaya} guarantees that $T' = A' \hash G$ is Azumaya. Indeed, $Z(A' \hash G) = \Bbbk[u+v, uv][(u-v)^{-2}]$ and
\begin{align*}
(u-v)^{-2} \Bigg( (u+v-\alpha)(u+v + \alpha) - 4\left(uv - \frac{\alpha^2}{4}\right) \hspace{-2pt} \Bigg) = (u-v)^{-2} \left( (u+v)^2 - 4uv \right) = 1
\end{align*}
so that $\MaxSpec Z(T') = \{ \langle u+v - \alpha, uv - \beta \rangle \mid \alpha^2 \neq 4 \beta \}$. As before, one can then show that $T'/\mathfrak{m}T' \cong M_2(\Bbbk)$ for all $\mathfrak{m} \in \MaxSpec Z(T')$, so that $T' = A' \hash G$ is Azumaya by Lemma \ref{azumayalem}.
\end{example}


This example demonstrates a strategy that we will utilise frequently in Section \ref{kleinazumaya}. It will often be the case that a generic maximal ideal of the centre of $A$ has trivial stabiliser, and we can remove those that do not by localising $A$ at a suitable subset of $Z(A)$. The new algebra $A'$ and the group $G$ will then meet the hypothesis of Theorem \ref{outerazumaya}, and so $A' \hash G$ will be Azumaya.


\section{Azumaya skew group algebras where a cyclic group acts inner} \label{innersec}

We now work towards proving Theorem \ref{innerintro}, before providing an application to motivate this result. Let $G$ be a cyclic group acting inner on an algebra $A$. As in the previous section, one would expect that $A$ must necessarily be Azumaya to ensure that $A \hash G$ is Azumaya. Under our hypotheses, this turns out to be both necessary and sufficient. \\
\indent We begin with a lemma that determines the centre of such a skew group ring.

\begin{lem} \label{innercentre}
Suppose that $G = \langle g \mid g^n \rangle$ acts inner on $A$, so there exists a unit $\eta \in A$ such that $g^i \cdot a = \eta^{i} a \eta^{-i}$ for all $a \in A$. Then
\begin{align*}
Z(A \hash G) = Z(A)[(\eta^{-1}g)^{\pm 1}].
\end{align*}
\end{lem}
\begin{proof}
First observe that if $z \in Z(A)$, $a \in A$ and $g^i \in C_n$ then
\begin{align*}
z \eta^{-1} g \cdot ag^i = z \eta^{-1} \eta a \eta^{-1} g g^i = z a \eta^{-1} g g^i = a z \eta^{-1} g^i g = a g^i \eta^{-i} z \eta^{-1} \eta^{i} g = ag^i \cdot z \eta^{-1} g,
\end{align*}
establishing the inclusion $Z(A)[(\eta^{-1}g)^{\pm 1})] \subseteq Z(A \hash G)$. For the reverse inclusion, consider any element $x = \sum_{i=0}^{n-1} a_i g^i \in Z(A \hash G)$. Then for any $b \in A$ we have
\begin{align*}
\sum_{i=0}^{n-1} b a_i g^i = bx = xb = \sum_{i=0}^{n-1} a_i g^i b = \sum_{i=0}^{n-1} a_i \eta^{i} b \eta^{-i} g^i,
\end{align*}
and so $ba_i =  a_i \eta^{i} b \eta^{-i}$ for each $i$. Equivalently, $b a_i \eta^{i} =  a_i \eta^{i} b$ for each $i$ and for each $b \in A$, and so each $a_i \eta^{i}$ lies in $Z(A)$. Therefore 
\begin{align*}
x = \sum_{i=0}^{n-1} (a_i \eta^{i}) \eta^{-i} g^i = \sum_{i=0}^{n-1} (a_i \eta^{i}) (\eta^{-1} g)^i \in Z(A)[(\eta^{-1}g)^{\pm 1})] \in Z(A)[(\eta^{-1}g)^{\pm 1})],
\end{align*}
as required.
\end{proof}

We can now prove the main result of this section:

\begin{thm} \label{innerazumaya}
Let $A$ be a prime $\Bbbk$-algebra and let $G = \langle g \mid g^n \rangle$ be a cyclic group acting inner on $A$, so there exists a unit $\eta \in A$ such that $g^i \cdot a = \eta^{i} a \eta^{-i}$ for all $a \in A$. Suppose also that $T = A \hash G$ is prime. Then $A$ is Azumaya if and only if $T$ is Azumaya, and in this case they have the same rank.
\end{thm}
\begin{proof}
$(\Rightarrow)$ Assume that $A$ is Azumaya of rank $d^2$. First observe that $T$ is PI: indeed, $A$ is necessarily module-finite over its centre, and since $Z(T) = Z(A)[(\eta^{-1}g)^{\pm 1}]$ where $(\eta^{-1} g)^n  =\eta^{-n}$, it follows that $T$ is also module-finite over its centre, which implies that $T$ is PI by \cite[Corollary 13.1.13 (iii)]{mandr}. Therefore, by Lemma \ref{azumayalem} it suffices to show that $T/\mathfrak{m}T \cong M_d(\Bbbk)$ for each $\mathfrak{m} \in \MaxSpec Z(T)$. \\
\indent To this end, let $\mathfrak{m}$ be a maximal ideal of $Z(T)$. 
Then $\mathfrak{m} \cap Z(A)$ is a maximal ideal of $Z(A)$ by \cite[Lemma 1.4.2]{benson}, and since $A$ is Azumaya, $(\mathfrak{m} \cap Z(A))A$ is a maximal ideal of $A$. Therefore, $B \coloneqq A/(\mathfrak{m} \cap Z(A))A \cong M_d(\Bbbk)$ which, in particular, is a simple ring. Let $\pi: A \to B$ be the natural projection and define a linear map 
\begin{align*}
\theta : B \to T/\mathfrak{m}T, \quad \theta(b) = a + \mathfrak{m}T, \text{ where } \pi(a) = b.
\end{align*}
This map is well-defined in the sense that it does not depend on the choice of the preimage of $b$: if $\pi(a) = b = \pi(a')$, then $a-a' \in (\mathfrak{m} \cap Z(A))A \subseteq \mathfrak{m}T$, so that $a + \mathfrak{m}T = a' + \mathfrak{m}T$. The map $\theta$ is also clearly a ring homomorphism, and we claim that it is surjective. To this end, let $t = \sum_{i=0}^{n-1} a_i g^i \in T$. Since $\Bbbk$ is algebraically closed and $A$ is a finitely generated $\Bbbk$-algebra, $\mathfrak{m}$ contains an element of the form $\eta^{-1}g - \lambda$, where $\lambda \in \Bbbk$. By replacing each instance of $g$ in $t$ by $(g - \lambda \eta) + \lambda \eta$ and noting that $g$ and $\lambda \eta$ commute, we can write $t = \sum_{i=0}^{n-1} a_i ' (g-\lambda \eta)^i$ for some $a_i' \in A$. Therefore $t + \mathfrak{m}T = a_0' + \mathfrak{m}T$ so that $\theta(\pi(a_0')) = a_0' + \mathfrak{m}T = t + \mathfrak{m}T$, and therefore $\pi$ is surjective. Since $B$ is simple and $\theta$ is not the zero map, $\ker \theta$ is trivial, and hence $\theta$ is an isomorphism. Therefore $T/\mathfrak{m}T \cong B \cong M_d(\Bbbk)$, as required. \\
\indent $(\Leftarrow)$ For the converse, suppose that $T$ is Azumaya of rank $d^2$. Since $T$ is PI of PI degree $d$, $A$ is PI and $\PIdeg A \leqslant d$. Using Lemma \ref{azumayalem}, it suffices to show that if $\mathfrak{m}$ is a maximal ideal of $Z(A)$ then $A/\mathfrak{m}A \cong M_d(\Bbbk)$, since this will force $\PIdeg A = d$. So let $\mathfrak{m} \in \MaxSpec Z(A)$. 
Now, $Z(A)$ and $Z(T) = Z(A)[(\eta^{-1}g)^{\pm 1}] \cong Z(A)[t]/\langle t^n - \eta^{-n} \rangle$ are finitely generated $\Bbbk$-algebras, so $Z(T)$ has a maximal ideal of the form $\mathfrak{m}' = \mathfrak{m} + \langle \eta^{-1} g - \beta \rangle$ for some $\beta \in \Bbbk^\times$. Since $T$ is Azumaya of rank $d^2$, $T/\mathfrak{m}'T \cong M_d(\Bbbk)$. Define a linear map 
\begin{align*}
\widetilde{\phi} : T \to A/\mathfrak{m}A, \quad \widetilde{\phi}(ag^i) = a (\eta \beta)^i + \mathfrak{m}A,
\end{align*}
which is easily checked to be a surjective ring homomorphism. It is clear that $\mathfrak{m}'T \subseteq \ker \widetilde{\phi}$, so we get a well-defined surjection $\phi : T/\mathfrak{m}'T \to A/\mathfrak{m}A$. But $T/\mathfrak{m}'T \cong M_d(\Bbbk)$ is a simple ring and $\phi$ is not the zero map, so $\ker \phi$ is trivial and $\phi$ is an isomorphism. Therefore $A/\mathfrak{m}A \cong T/\mathfrak{m}'T \cong M_d(\Bbbk)$, as required.
\end{proof}

\begin{rem}
It is not automatically the case that $T$ is prime under the above hypotheses, even if the action is faithful. For example, set $A = M_2(\mathbb{C})$ and $G = C_2 = \langle g \rangle$, where $g$ acts as conjugation by $\begin{psmallmatrix}0&1\\-1&0\end{psmallmatrix}$. Then
\begin{align*}
x = \begin{pmatrix} 1 & 0 \\ 0 & 1 \end{pmatrix} + \begin{pmatrix} 0 & i \\ -i & 0 \end{pmatrix} \hspace{-1pt} g, \quad y = \begin{pmatrix} 1 & 0 \\ 0 & 1 \end{pmatrix} + \begin{pmatrix} 0 & -i \\ i & 0 \end{pmatrix} \hspace{-1pt} g
\end{align*}
are central elements of $T \coloneqq A \hash G$ which satisfy $xy = 0$. Therefore $xTy = xyT = 0$ and so $T$ is not prime.
\end{rem}

\begin{example}
We provide an application of Theorem \ref{innerazumaya} which will be generalised in Section \ref{kleinazumaya}. Let $A = \Bbbk_{-1}[u^{\pm 1},v^{\pm 1}]$ and $G = C_2$, where the generator $g$ acts via $g \cdot u = -u$, $g \cdot v = -v$. This is an inner action, since $g$ acts as conjugation by $(uv)^{-1}$. We claim that $A \hash G$ is Azumaya; since $A$ is Azumaya, by Theorem \ref{innerazumaya} we only need to show that $A \hash G$ is prime. \\
\indent By \cite[Lemma 6.17]{gw}, it suffices to show that $Q(A \hash G) \cong \Bbbk_{-1}(u,v) \hash G$ is simple. Since $\Bbbk_{-1}(u,v)$ is simple, this is the case if and only if $Z(\Bbbk_{-1}(u,v) \hash G)$ is a field, using \cite[Theorem 1.2]{oinert}. Observing that $g$ acts as conjugation by $(uv)^{-1}$, Lemma \ref{innercentre} implies that
\begin{align*}
Z(\Bbbk_{-1}(u,v) \hash G) = \Bbbk(u^2,v^2)[(uvg)^{\pm 1}] \cong \frac{\Bbbk(u^2,v^2)[t]}{\langle t^2-u^2v^2 \rangle}. 
\end{align*}
Since $t^2 - u^2v^2$ is irreducible over $\Bbbk[u^2,v^2][t]$, it is also irreducible over $\Bbbk(u^2,v^2)[t]$ by Gauss's Lemma. Therefore  $Z(\Bbbk_{-1}(u,v) \hash G)$ is a field, which implies that $A \hash G$ is prime.\\
\indent To close this example, note that $\Bbbk_{-1}[u^{\pm 1}, v]$ is not Azumaya, but one can show that the skew group ring $\Bbbk_{-1}[u^{\pm 1}, v] \hash C_2$ is Azumaya. This does not contradict Theorem \ref{innerazumaya}, since in this case the action is not inner, as $v$ is not invertible.
\end{example}


\section{Examples coming from quantum Kleinian singularities} \label{kleinazumaya}

We now apply the results of the previous two sections to some examples. For each quantum Kleinian singularity $A^G$, the ring $A$ is not Azumaya, so to use our results to show that $A \hash G$ is Azumaya we must first localise $A$ at an appropriate Ore set. Moreover, when we seek to apply Theorem \ref{outerazumaya}, it is often the case that the action does not act freely on $Z(A)$, and so in these cases we localise $A$ again to remove any maximal ideals with nontrivial stabiliser, as in Example \ref{outerexample}. \\
\indent To be able to apply the results in the previous sections, we require a number of properties of the skew group rings $A \hash G$, which we collect in Lemma \ref{AGprime}. In particular, we show that these rings are prime, a property which is already known and follows from \cite[Lemma 3.10]{bhz} and the fact that Auslander's Theorem holds in this setting. However, we wish to give an alternative proof of Auslander's Theorem, and so we cannot appeal to this result. \\
\indent We first determine the symmetric Martindale ring of quotients of the Jordan plane.

\begin{lem} \label{martindalejordan}
Let $A = \Bbbk_J[u,v] = \Bbbk \langle u,v \rangle / \langle uv - vu - u^2 \rangle$. Then $Q_s(A) = A[u^{-1}]$.
\end{lem}
\begin{proof}
The fact that $A[u^{-1}] \subseteq Q_s(A)$ follows from \cite[Lemma 2.1 (i)]{passman2}. For the other inclusion, let $q \in Q_s(A)$, and write $q = ab^{-1}$, where $a,b \in A$ and $b \neq 0$. By (1), there exists a nonzero ideal $I$ of $A$ such that $qI \subseteq R$. By \cite[Theorem 5.1]{irving2}, every nonzero two-sided ideal of $A$ contains $u^k$ for some $k \geqslant 0$. Then $ab^{-1}u^k \in A$ forcing $b$ to be a power of $u$ as well. That is, $q \in A[u^{-1}]$ and so $Q_s(A) \subseteq A[u^{-1}]$. 
\end{proof}

\begin{lem} \label{AGprime}
Let $A^G$ be a quantum Kleinian singularity. Then $T \coloneqq A \hash G$ is a prime noetherian finitely generated $\Bbbk$-algebra which is Cohen-Macaulay and of GK dimension 2.
\end{lem}
\begin{proof}
The only property that does not immediately follow from results in the literature is that $T$ is prime, so we establish the other properties first. Since $A$ is either $\Bbbk_q[u,v]$ or $\Bbbk_J[u,v]$, both of which are noetherian, $T$ is noetherian by \cite[Lemma 1.5.11]{mandr}. Finite generation is clear, while $\GKdim T = \GKdim A = 2$ by \cite[Proposition 8.2.9]{mandr}, and the Cohen-Macaulay property follows from \cite[Proposition 3.3]{bhz}. To show that $T$ is prime, we consider each case separately. By \cite[Corollary 12.6]{passman}, it suffices to show that $A \hash G_{X\text{-inn}}$ is prime. \\
\indent We first consider case (i). By \cite[Proposition 2.2]{bergen}, every $X$-inner automorphism of $A$ is given by conjugation by a monomial in $u$ and $v$. In particular, if $q$ is not a root of unity, then $G_{X\text{-inn}}$ is trivial, whence the result. So assume that $q$ has order $k$ and write $\ell = \lcm(n,k)$.
If $g^i$ acts $X$-inner, say as conjugation by some monomial $f \in \Bbbk_q(u,v)$, then since $u^k$ is central we have
\begin{align*}
\omega^{ik} u^k = g^i \cdot u^k = f u^k f^{-1} = u^k.
\end{align*}
Therefore $ik \in n \mathbb{Z}$, so that $i \tfrac{k/\gcd(n,k)}{n/\gcd(n,k)} \in \mathbb{Z}$, and so necessarily $i \in \tfrac{n}{\gcd(n,k)} \mathbb{Z} = \tfrac{\ell}{k} \mathbb{Z}$. But also $g^{\ell/k} \cdot a = (uv)^{\ell/n} a (uv)^{-\ell/n}$ for all $a \in \Bbbk_q(u,v)$, so that $G_{X\text{-inn}} = \langle g^{\ell/k} \rangle$. We therefore need to show that $\Bbbk_q[u,v] \hash \langle g^{\ell/k} \rangle$ is prime, and by \cite[Lemma 6.17]{gw} it suffices to show that $Q(T) \cong \Bbbk_q(u,v) \hash \langle g^{\ell/k} \rangle$ is simple. Since $\Bbbk_q(u,v)$ is simple, by \cite[Theorem 1.2]{oinert} we only need to show that $Z(\Bbbk_q(u,v) \hash \langle g^{\ell/k} \rangle)$ is a field. Since $g^{\ell/k}$ acts inner on $\Bbbk_q(u,v)$ as conjugation by $(uv)^{-\ell/n}$, by Lemma \ref{innercentre} we have
\begin{align*}
Z(\Bbbk_q(u,v) \hash \langle g^{\ell/k} \rangle) = Z(\Bbbk_q(u,v)[(uv)^{-\ell/n} g^{\ell/k}] = \Bbbk(u^k, v^k)[(uv)^{-\ell/n} g^{\ell/k}] \cong \frac{\Bbbk(u^k, v^k)[t]}{\langle t^h - u^{-k} v^{-k} \rangle}.
\end{align*}
Now, $t^h - u^{-k} v^{-k}$ is irreducible over $\Bbbk[u^{-k}, v^{-k}][t]$, and so Gauss's Lemma implies that it is also irreducible over $\Bbbk(u^k, v^k)[t]$. Therefore $\Bbbk(u^k, v^k)[t]/\langle t^h - u^{-k} v^{-k} \rangle$ is a field, and hence so too is $Z(\Bbbk_q(u,v) \hash \langle g^{\ell/k} \rangle)$, from which it follows that $T$ is prime. \\
\mbox{} \hspace{10pt} Now consider case (ii). Since the action of the generator of $S_2$ interchanges $u$ and $v$, it quickly follows from degree considerations that $G_{X\text{-inn}}$ is trivial, and so $T$ is prime since $\Bbbk_{-1}[u,v]$ is prime. \\
\mbox{} \hspace{10pt} For case (iii), the analysis in the previous two paragraphs shows that 
\begin{align*}
\Bbbk_{-1}[u,v] \hash G_{X\text{-inn}} = \left \{
\begin{array}{c l}
\Bbbk_{-1}[u,v] \hash \langle g^{n/2} \rangle & \text{if } n \text{ is even} \\
\Bbbk_{-1}[u,v] & \text{if } n \text{ is odd} 
\end{array}
\right . .
\end{align*}
In either case, we already know that $A \hash G_{X\text{-inn}}$ is prime, and so $T$ is prime. \\
\indent Finally, consider case (iv). Since $Q_s(\Bbbk_J[u,v]) = \Bbbk_J[u^{\pm 1}, v]$, every $X$-inner automorphism is given by conjugation by a power of $u$. But 
\begin{align*}
u^k u u^{-k} = u \neq -u = g \cdot u
\end{align*} 
and so $G_{X\text{-inn}}$ is trivial, from which it follows that $T$ is prime.
\end{proof}

\indent We now show that, for case (i) when $q$ is a root of unity, a suitable localisation of $A \hash G = \Bbbk_q[u,v] \hash C_n$ is Azumaya. Since $A$ is not Azumaya, we instead consider $A' = \Bbbk_q[u^{\pm 1},v^{\pm 1}]$ and write $T' \coloneqq A' \hash G$. Note that $T' = T[u^{-1}, v^{-1}]$, which is prime since $T$ is prime. By combining Theorems \ref{outerazumaya} and \ref{innerazumaya}, we now show that $T'$ is Azumaya.

\begin{prop} \label{bisazumaya}
Suppose that $q$ is a $k$th root of unity. Then the algebra $T' \coloneqq \Bbbk_{q}[u^{\pm 1}, v^{\pm 1}] \hash C_n$ is Azumaya.
\end{prop}
\begin{proof}
Throughout the proof, it will be convenient to write $\ell = \lcm(n,k)$. We also let $\varepsilon$ be a primitive $\ell$th root of unity, so that we may as well assume that $\omega = \varepsilon^{\ell/n}$ and $q = \varepsilon^{\ell/k}$. \\
\indent Note that $g^{\ell/k}$ acts inner on $A'$, since
\begin{gather*}
g^{\ell/k} \cdot u = \omega^{\ell/k} u = \varepsilon^{\ell^2/nk} u = q^{\ell/n} u = (uv)^{\ell/n} u \gap (uv)^{-\ell/n} \\
g^{\ell/k} \cdot v = \omega^{-\ell/k} v = \varepsilon^{-\ell^2/nk} v = q^{-\ell/n} v = (uv)^{\ell/n} v \gap (uv)^{-\ell/n}
\end{gather*}
and so the subgroup $\langle g^{\ell/k} \rangle$ acts inner on $A'$. By Lemma \ref{AGprime}, $A \hash \langle g^{\ell/k} \rangle$ is prime, and hence so too is the localisation $A' \hash \langle g^{\ell/k} \rangle$. Therefore Theorem \ref{innerazumaya} ensures that $A' \hash \langle g^{\ell/k} \rangle$ is Azumaya. \\
\indent Since $C_n/\langle g^{\ell/k} \rangle \cong C_{\ell/k}$, as in \cite[Lemma 1.5.9]{mandr} we may define an isomorphism
\begin{align*}
T' \cong (A' \hash \langle g^{\ell/k} \rangle) * C_{\ell/k}.
\end{align*}
When defining this isomorphism, we think of $C_{\ell/k}$ as the group $\langle \sigma \mid \sigma^{\ell/k} \rangle$ and write $\overline{C_{\ell/k}} = \{ \overline{\sigma^i} \mid 0 \leqslant i < \ell/k \}$, where $\overline{\sigma^i} = g^i$. The induced automorphisms of $A' \hash \langle g^{\ell/k} \rangle$ are given by
\begin{align*}
\alpha_{\sigma^i}(u) = \omega^i u, \quad \alpha_{\sigma^i}(v) = \omega^{-i} v, \quad \alpha_{\sigma^i}(g^{\ell/k}) = g^{\ell/k}.
\end{align*}
We wish to apply Theorem \ref{outerazumaya} to show that $(A' \hash \langle g^{\ell/k} \rangle) * C_{\ell/k}$ is Azumaya. We first show that the action of $C_{\ell/k}$ on $R \coloneqq A' \hash \langle g^{\ell/k} \rangle$ is $X$-outer, so suppose that $\alpha_{\sigma^r}$ acts as conjugation by an element of $Q_s(R)$. Since $R$ is prime and PI, $Q_s(R) = R(Z(R)^\times)^{-1}$. Moreover, since central elements do not affect the result when conjugating, we can assume that $\alpha_{\sigma^r}$ is given by conjugation by an element of the form $\sum_t \lambda_t u^{i_t} v^{j_t} g^{m_t \ell/k}$, where $i_t, j_t, m_t \in \mathbb{Z}$ and $\lambda_t \in \Bbbk$. We therefore have
\begin{align*}
\alpha_{\sigma^r}(u) \cdot \sum_t \lambda_t u^{i_t} v^{j_t} g^{m_t \ell/k} = \sum_t \lambda_t u^{i_t} v^{j_t} g^{m_t \ell/k} \cdot u,
\end{align*}
or equivalently,
\begin{align*}
\sum_t \omega^r \lambda_t u^{i_t+1} v^{j_t} g^{m_t \ell/k} = \sum_t q^{j_t} \omega^{m_t \ell/k} \lambda_t u^{i_t+1} v^{j_t} g^{m_t \ell/k}.
\end{align*}
Therefore, for each $t$ we require
\begin{align*}
\varepsilon^{r \ell/n} = \varepsilon^{j_t \ell/k + m_t \ell^2/nk} \quad \Rightarrow \quad j_t \tfrac{\ell}{k} + m_t \tfrac{\ell^2}{nk} - r \tfrac{\ell}{n} \in \ell \mathbb{Z}.
\end{align*}
This forces $j_t + m_t \tfrac{\ell}{n} - r \tfrac{k}{n} \in k \mathbb{Z}$, and so necessarily $r \tfrac{k}{n} = r \tfrac{k/\gcd(n,k)}{n/\gcd(n,k)}$ is an integer. Since $k/\gcd(n,k)$ and $n/\gcd(n,k)$ are coprime, we must then have $r \in \tfrac{n}{\gcd(n,k)} \mathbb{Z} = \tfrac{\ell}{k}\mathbb{Z}$. But $0 \leqslant r < \ell/k$ so $r=0$, which means that the induced action is $X$-outer. \\
\indent We must also show that $C_{\ell/k}$ acts as a group on
\begin{align*}
Z(R) = \Bbbk[u^{\pm k}, v^{\pm k}, ((uv)^{-\ell/n} g^{\ell/k})^{\pm 1}] = \Bbbk[u^{\pm k}, ((uv)^{-\ell/n} g^{\ell/k})^{\pm 1}].
\end{align*}
Note that the action of $\alpha_{\sigma^i}$ on the generators $u^k$ and $(uv)^{-\ell/n}g^{\ell/k}$ of the centre (where here $0 \leqslant i < \ell/k$) is given by
\begin{align*}
\alpha_{\sigma^i}(u^k) = \omega^{ik} u^k, \quad \alpha_{\sigma^i}((uv)^{-\ell/n}g^{\ell/k}) = (uv)^{-\ell/n}g^{\ell/k},
\end{align*}
and so $\alpha_e$ is the identity, and $\alpha_{\sigma^i} \alpha_{\sigma^j}$ and $\alpha_{\sigma^{i+j \bmod \ell/k}}$ give rise to the same automorphism of $Z(R)$. \\
\indent Finally, we need to show that $\Stab_{C_{\ell/k}} (\mathfrak{m}) $ is trivial for each $\mathfrak{m} \in \MaxSpec Z(R)$. Note that $\mathfrak{m}$ has form
\begin{align*}
\langle u^k - \alpha, (u v)^{-\ell/n} g^{\ell/k} - \beta \rangle,
\end{align*}
where $\alpha \neq 0 \neq \beta$. Therefore, to show that any such ideal has trivial stabiliser, it suffices to show that if $0 \leqslant r < \ell/k$ and $\alpha_{\sigma^r}(u^k) = u^k$, then $r=0$. If $\alpha_{\sigma^r}(u^k) = u^k$, then $\varepsilon^{r \ell k / n} = 1$, and so we require $\tfrac{rk}{n} \in \mathbb{Z}$. But, as before in the proof, this forces $r \in \tfrac{\ell}{k} \mathbb{Z}$ and hence $r=0$, and so every maximal ideal has trivial stabiliser. Applying Theorem \ref{outerazumaya}, we find that $(\Bbbk_{q}[u^{\pm 1}, v^{\pm 1}] \hash \langle g^{\ell/k} \rangle) * C_{\ell/k}$ is Azumaya, and hence so too is $T'$.
\end{proof}

\indent We now consider case (ii). It is easy to show that the action in this case is $X$-outer, and so we wish to apply Theorem \ref{outerazumaya}. As with case (i), we replace $A = \Bbbk_{-1}[u,v]$ by the Azumaya algebra $A' = \Bbbk_{-1}[u^{\pm 1}, v^{\pm 1}]$. We will also need to localise a second time to ensure that every maximal ideal of the centre has trivial stabiliser, as in Example \ref{outerexample}.

\begin{prop} \label{iiiazumaya}
Write $A' = \Bbbk_{-1}[u^{\pm 1}, v^{\pm 1}]$ and $A'' = A'[(u^2-v^2)^{-1}]$. Then the algebra $T'' \coloneqq A'' \hash S_2$ is Azumaya.
\end{prop}
\begin{proof}
Seeking to apply Theorem \ref{outerazumaya}, first note that $A'$ is a prime noetherian $\Bbbk$-algebra which is Azumaya. Localisation preserves these properties, so the same is true of $A''$. Since the generator of $S_2$ interchanges $u$ and $v$, degree considerations imply that the action is $X$-outer, so it remains to show that $G$ acts freely on $\MaxSpec Z(A'')$. Now, using the fact that $Z(RX^{-1}) = Z(R)X^{-1}$ for a noetherian ring $R$ and a multiplicative set $X$ of regular elements contained in $Z(R)$, we find that
\begin{align*}
Z(A'') = \Bbbk[u^{\pm 2}, v^{\pm 2}] [(u^2-v^2)^{-1}] = Z(A)[(u^2-v^2)^{-1}].
\end{align*} 
The maximal ideals of $Z(A'')$ are in one-to-one correspondence with maximal ideals of $Z(A')$ which do not contain $u^{2}-v^{2}$ \cite[Theorem 10.20]{gw}. But $\MaxSpec Z(A') = \{ \langle u^2-\alpha, v^2-\beta \rangle \mid \alpha \neq 0 \neq \beta \}$, and such a maximal ideal $\mathfrak{m}$ has nontrivial stabiliser if and only if $\alpha = \beta$ and this happens if and only if $u^2 - v^2$ lies in $\mathfrak{m}$. Therefore such $\mathfrak{m}$ do not give rise to maximal ideals of $Z(A'')$, so the stabiliser of every maximal ideal of $Z(A'')$ is trivial, and so Theorem \ref{outerazumaya} tells us that $T''$ is Azumaya.
\end{proof}

\indent The final case we consider is case (iii). When showing that a suitable localisation of $A \hash G$ is Azumaya in these cases, the set at which we localise depends on the parity of $n$, so we consider two separate cases. We first consider the case when $n$ is odd, since this is easier due to the action being $X$-outer.

\begin{prop} \label{disazumayaodd}
Suppose that $n$ is odd, and let $A' = \Bbbk_{-1}[u^{\pm 1}, v^{\pm 1}]$ and $A'' =  A'[(u^{2n}-v^{2n})^{-1}]$. Then $T'' \coloneqq A'' \hash D_n$ is Azumaya.
\end{prop}
\begin{proof}
We check that the hypotheses of Theorem \ref{outerazumaya} are met. Firstly, as in the proof of Proposition \ref{bisazumaya}, the nontrivial rotations $g^i$ act $X$-outer since the order of $q=-1$ is coprime to $n$, and each of the reflections $g^i h$ acts $X$-outer by degree considerations. Moreover, $A''$ is a prime noetherian Azumaya $\Bbbk$-algebra, so it remains to check that $G$ acts freely on maximal ideals of $Z(A'')$. These are in one-to-one correspondence with the maximal ideals of $Z(A')$ which do not contain $u^{2n}-v^{2n}$ \cite[Theorem 10.20]{gw}. But $\MaxSpec Z(A') = \{ \langle u^2-\alpha, v^2-\beta \rangle \mid \alpha \neq 0 \neq \beta \}$, and such a maximal ideal $\mathfrak{m}$ has nontrivial stabiliser if and only if $\alpha = \omega^{i} \beta$ for some $i$, and this happens if and only if $u^2 - \omega^i v^2$ lies in $\mathfrak{m}$. However, since $\prod_{0 \leqslant i < n} u^2 - \omega^i v^2 = u^{2n} - v^{2n}$, such $\mathfrak{m}$ do not give rise to maximal ideals of $Z(A'')$. Therefore the stabiliser of every maximal ideal of $Z(A'')$ is trivial, and so Theorem \ref{outerazumaya} tells us that $T''$ is Azumaya.
\end{proof}

\indent We now assume that $n$ is even, in which case the action is not $X$-outer, and so we have to combine Theorems \ref{outerazumaya} and \ref{innerazumaya} as in the proof of Proposition \ref{bisazumaya}.

\begin{prop} \label{disazumayaeven}
Suppose that $n$ is even, and let $A' = \Bbbk_{-1}[u^{\pm 1}, v^{\pm 1}]$. Then $T' \coloneqq A' \hash D_n$ is Azumaya.
\end{prop}
\begin{proof}
Write $m = n/2$. First note that we have an isomorphism
\begin{align*}
T' \cong (A' \hash \langle g^m \rangle ) * D_m,
\end{align*}
where, as shown in the proof of Proposition \ref{bisazumaya}, $A' \hash \langle g^m \rangle$ is a prime noetherian Azumaya algebra with centre $Z = \Bbbk[u^{\pm 2},v^{\pm 2}, uvg] \cong \Bbbk[x^{\pm 1},y^{\pm 1}, z^{\pm 1}]/\langle xy + z^2 \rangle$. Here, we think of $D_m$ as the group $\langle \sigma, \tau \mid \sigma^m, \tau^2, \tau \sigma = \sigma^{-1} \tau \rangle = \{ \sigma^i \tau^j \mid 0 \leqslant i < m, j =0,1 \}$, and our copy of the set $D_m$ inside $T'$ is given by $\overline{\sigma^i \tau^j} = g^i h^j$. Tracing through this isomorphism, one finds that the induced automorphisms of $A' \hash \langle g^m \rangle$ are given by
\begin{gather*}
\alpha_{\sigma^i}(u) = \omega^i u, \quad \alpha_{\sigma^i}(v) = \omega^{-i} v, \quad \alpha_{\sigma^i}(g^m) = g^m, \\
\alpha_{\sigma^i \tau}(u) = -\omega^i v, \quad \alpha_{\sigma^i \tau}(v) = \omega^i u, \quad \alpha_{\sigma^i \tau}(g^m) = g^m.
\end{gather*}
The same argument as in the proof of Proposition \ref{bisazumaya} shows that each of these automorphisms is $X$-outer. Moreover, the restrictions of the above automorphisms to $Z$ are given by
\begin{gather*}
\alpha_{\sigma^i}(x) = \omega^{2i} x, \quad \alpha_{\sigma^i}(y) = \omega^{-2i} y, \quad \alpha_{\sigma^i}(z) = z \\
\alpha_{\sigma^i \tau}(x) = \omega^{-2i} y, \quad \alpha_{\sigma^i \tau}(y) = \omega^{2i} x, \quad \alpha_{\sigma^i \tau}(z) = -z,
\end{gather*}
which is a group action. Finally, if $\mathfrak{m} = \langle u^2 - \alpha, v^2 - \beta, uvg - \gamma \rangle$, where $\alpha \neq 0 \neq \beta$ and $\alpha \beta + \gamma^2 = 0$, is a maximal ideal of $Z$, then
\begin{align*}
\alpha_{\sigma^i \tau^j}(\mathfrak{m}) = \left \{
\begin{array}{cl}
\langle u^2 - \omega^{-2i}\alpha, v^2 - \omega^{2i} \beta, uvg - \gamma \rangle & \text{if } j=0 \\
\langle u^2 - \omega^{-2i} \beta, v^2 - \omega^{2i} \alpha, uvg + \gamma \rangle & \text{if } j=1 \\
\end{array}
\right .
\end{align*}
which in either case is not equal to $\mathfrak{m}$ unless $i=0=j$. Therefore $D_m$ acts freely on $Z(A' \hash \langle g^m \rangle)$ so, by Theorem \ref{outerazumaya}, $ (\Bbbk_{-1}[u^{\pm 1}, v^{\pm 1}] \hash \langle g^m \rangle ) * D_m$ is Azumaya, and hence so too is $T'$.
\end{proof}


\section{Quantum Kleinian singularities are maximal orders}
As discussed in the introduction, our motivation for proving the results in the preceding sections was to prove the following result:

\begin{thm} \label{maxorderthm}
Suppose that $A^G$ is a quantum Kleinian singularity. Then $A \hash G$ is a maximal order.
\end{thm}

The above can be viewed as analogue of a result in \cite[Lemma 1.3]{cbh}, where they prove the result for case (0) by appealing to \cite[Theorem 4.6]{martin}, which requires $A$ to be commutative and so does not apply in our setting. On the other hand, \cite[Theorem 3.13]{martin} allows one to show that skew group rings of certain noncommutative rings are maximal orders, but this also requires the action to be $X$-outer, which is frequently not the case for the algebras of interest to us. Instead, we developed a case-by-case approach which uses the results of the previous sections. As an immediate corollary, we have the following:

\begin{cor} \label{invmaxorder}
Suppose that $A^G$ is a quantum Kleinian singularity. Then $A^G$ is a maximal order.
\end{cor}
\begin{proof}
Letting $e = \tfrac{1}{|G|}\sum_{g \in G} g$, we have $e(A \hash G)e \cong A^G$ by \cite[Lemma 3.1]{bhz}. The result now follows from \cite[Corollary 1.7]{marubayshi}.
\end{proof}

\subsection{Preliminary results}
We first recall and prove some results which will be used in the proof of Theorem \ref{maxorderthm}. We have the following result, which we state in a form most suited to our use:

\begin{thm}[{\cite[Theorem 3.13]{martin}}] \label{martinthm}
Let $A$ be a prime noetherian ring and $G$ a finite group acting on $A$ such that the action of $G$ is $X$-outer. Write $T = A \hash G$ and write
\begin{align*}
\Omega = \bigg \{ \bigcap_{g \in G} g \cdot p_0 \; \Big | \; p_0 \text{ is a reflexive height } 1 \text{ prime ideal of } A \bigg\}.
\end{align*}
Suppose also that the following two conditions hold:
\begin{enumerate}[{\normalfont (1)},leftmargin=30pt,topsep=0pt,itemsep=0pt]
\item $A$ is a maximal order in its quotient ring; and
\item $p T \in \Spec T$ for all $p \in \Omega$.
\end{enumerate}
Then $T$ is a prime maximal order.
\end{thm}

\indent It will turn out that the above result is of limited use to us, since our actions are frequently not $X$-outer. Instead, we will use Lemma \ref{maxorderlem} and the following result to show that the rings $T = A \hash G$ are maximal orders.

\begin{lem} \label{height2}
Let $T$ be a finitely generated prime noetherian PI $\Bbbk$-algebra which is Cohen-Macaulay and of GK dimension $d$. Then the (classical) Krull dimension of $T$ is $d$ and $O_\ell(P) = T = O_{r}(P)$ for every prime ideal $P$ of $T$ with $\height P \geqslant 2$.
\end{lem}
\begin{proof}
That $\clKdim = d$ under these hypotheses is well-known; see \cite[Theorem 10.10]{lenagan}. If $P$ is a height $r$ prime of $T$ with $r \geqslant 2$, then $T/P$ is a finitely generated prime PI algebra and so by \cite[Theorem 10.10]{lenagan}, we find that $\GKdim T/P = \clKdim T/P = d-r$. Now, applying the functor $\Hom_T(-,T)$ to the short exact sequence of right $T$-modules
\begin{align*}
0 \to P \to T \to T/P \to 0,
\end {align*}
we obtain the following portion of a long exact sequence:
\begin{align*}
\Hom_T(T/P,T) \to T \to \Hom_T(P,T) \to \Ext_T^1(T/P,T).
\end {align*}
By the Cohen-Macaulay property of $T$, the grade of $T/P$ satisfies
\begin{align*}
j(T/P) = \GKdim(T) - \GKdim(T/P) = d - (d-r) = r \geqslant 2.
\end{align*}
In particular, $\Hom_T(T/P,T) = 0 = \Ext_T^1(T/P,T)$, and so $\Hom_T(P,T) = T$. Therefore,
\begin{align*}
T \subseteq \End_T(P) \subseteq \Hom_T(P,T) = T,
\end{align*}
and so $O_\ell(P) = \End_T(P) = T$. The proof that $O_r(P) = T$ is symmetrical.
\end{proof}

\indent We now explain why the above is relevant. If $A^G$ is a quantum Kleinian singularity, then Lemma \ref{AGprime} tells us that $T = A \hash G$ is a prime noetherian $\Bbbk$-algebra which is Cohen-Macaulay. Moreover, other than for case (i) when $q$ is not a root of unity and for case (iv), $T$ is PI since it is finite over its centre. Finally, each such skew group ring is a noetherian finitely generated $\Bbbk$-algebra and so by \cite[Theorem 10.10]{lenagan} its (classical) Krull dimension is equal to its GK dimension, namely 2. Therefore, since Lemma \ref{maxorderlem} tells us that $T$ is a maximal order provided that $\End_T(P_T) = T = \End_T({}_T P)$ for all prime ideals $P$, Lemma \ref{height2} implies that we only need to check that $\End_T(P) = T$ for all height 1 primes $P$ of $T = A \hash G$. \\
\indent We now seek to prove Theorem \ref{maxorderthm} for each case in turn, beginning with the cases where the conditions of Theorem \ref{martinthm} are easily verified.

\subsection{Proof of Theorem \ref{maxorderthm} for case (iv)}
That $A \hash G$ is a maximal order in this case follows relatively quickly from Theorem \ref{martinthm} since, as we saw in the proof of Lemma \ref{AGprime} the action in this case is $X$-outer.

\begin{thm} \label{maxordercaseh}
The algebra $T = \Bbbk_J[u,v] \hash C_2$ in case \emph{(iv)} is a maximal order.
\end{thm}
\begin{proof}
First recall that $A = \Bbbk_J[u,v]$ is a noetherian domain. By \cite[Theorem 2.10]{stafford2}, $\Bbbk_J[u,v]$ is a maximal order, and so condition (1) of Theorem \ref{martinthm} is satisfied. Moreover, by by \cite[Theorem 5.2]{irving2}, the only height one prime of $\Bbbk_J[u,v]$ is $\mathfrak{p} = \langle u \rangle$, which is also reflexive and $G$-stable. We claim that $\mathfrak{p} T$ is a prime ideal of $T$, or equivalently that $T/\mathfrak{p} T$ is prime. Since $\mathfrak{p}$ is $G$-stable, this latter ring is isomorphic to $(\Bbbk_J[u,v]/\mathfrak{p}) \hash C_2 \cong \Bbbk[v] \hash C_2$, and so by \cite[Lemma 6.17]{gw} it suffices to show that $Q(\Bbbk[v] \hash C_2)$ is a simple ring. But $Q(\Bbbk[v] \hash C_2) \cong \Bbbk(v) \hash C_2$, and this ring is simple by \cite[Proposition 7.8.12]{mandr}. Thus $\mathfrak{p} T$ is a prime ideal of $T$. Therefore condition (2) of Theorem \ref{martinthm} holds, so $T$ is a maximal order.
\end{proof}

\subsection{Proof of Theorem \ref{maxorderthm} for case (i)}
By the proof of Lemma \ref{AGprime}, in this case the action is $X$-outer if and only if $q$ is not a root of unity or when the orders of $g$ and $q$ are coprime. We first consider the former case:

\begin{thm} \label{maxordercaseb}
Suppose that $q$ is not a root of unity. Then the algebra $T = \Bbbk_q[u,v] \hash C_n$ in case \emph{(i)} is a maximal order.
\end{thm}
\begin{proof}
It is well-known that $A = \Bbbk_q[u,v]$ is a noetherian domain. Moreover, since it is Artin-Schelter regular, it follows from \cite[Theorem 2.10]{stafford2} that it is a maximal order, and so condition (1) of Theorem \ref{martinthm} is satisfied. \\
\indent We now show that condition (2) holds. Since $q$ is not a root of unity,  by \cite[Exercise 10P]{gw} the height 1 prime ideals of $A$ are $\langle u \rangle$ and $\langle v \rangle$. Moreover, these ideals are reflexive and $G$-stable, so $\Omega = \{ \langle u \rangle, \langle v \rangle \}$. The same argument as in the proof of Theorem \ref{maxordercaseh} shows that both $\langle u \rangle T$ and $\langle v \rangle T$ are prime ideals of $T$, so condition (2) of Theorem \ref{martinthm} holds. Therefore $T$ is a maximal order.
\end{proof}

We now turn our attention to the case where $q$ is a root of unity. Despite the fact that the action is $X$-outer when the order of $q$ is coprime to $n$, it is difficult to check condition (2) of Theorem \ref{martinthm} in this case. This is due to the fact that $\Bbbk_q[u,v]$ has many more $G$-prime ideals than when $q$ is not a root of unity, see \cite[\S 8]{irving1}. We therefore take advantage of the fact that $\Bbbk_q[u^{\pm 1}, v^{\pm 1}] \hash C_n$ is Azumaya. 

\begin{thm} \label{bmaxorder}
Suppose that $q$ is a $k$th root of unity. Then the algebra $T = \Bbbk_q[u,v] \hash C_n$ in case \emph{(i)} is a maximal order.
\end{thm}
\begin{proof}
By the discussion after Lemma \ref{height2}, it suffices to show that $O_\ell(P) = T = O_r(P)$ for all height $1$ primes. \\
\indent So suppose that $P$ has height $1$. If $P$ contains $u$, then since $\langle u \rangle$ is a prime ideal of $T$  we have $P = \langle u \rangle$ (that $\langle u \rangle$ is prime follows from the same argument as in the proof of Theorem \ref{maxordercaseb}, which does not make use of $q$ having infinite order). Since $u$ is a normal nonzerodivisor, $O_\ell(P) = T$. Similarly, if $P$ contains $v$ then $P = \langle v \rangle$ is a prime ideal of $T$ and $O_\ell(P) = T$. Write $A' \coloneqq \Bbbk_q[u^{\pm 1},v^{\pm 1}]$ and $T' \coloneqq A' \hash G$, both of which are Azumaya by Proposition \ref{bisazumaya}. Now let $P$ be a height $1$ prime of $T$ which does not contain $u$ or $v$, so that $P$ corresponds to a height 1 prime $PT'$ of $T'$. Since $T'$ is Azumaya, there exists a height 1 prime $\mathfrak{p}$ of $Z(T')$ such that $PT' = \mathfrak{p} T'$. 
Writing $\ell = \lcm(n,k)$, it is not difficult to show that
\begin{align*}
Z(T') = \Bbbk[u^{\pm \ell}, (uv)^{-\ell/n} g^{\ell/k})^{\pm 1} ],
\end{align*}
which is a Laurent polynomial ring in two variables and is hence a UFD. Therefore its height 1 primes are principal, and so $\mathfrak{p} = z Z(T')$ for some $z \in Z(T')$. Since $z$ is a central nonzerodivisor, we have $\End_{T'}(zT') = T'$, and therefore we have a chain of inclusions
\begin{align*}
T \subseteq \End_T(P) \subseteq \End_{T'}(PT') = \End_{T'}(zT') = T'.
\end{align*}
It remains to show that this forces $\End_T(P) = T$. To this end, let $t' \in \End_T(P) \subseteq T'$ and choose $i \geqslant 0$ minimal such that $t \coloneqq (uv)^i t' \in T$. We claim that $i=0$, forcing $t' \in T$. Seeking a contradiction, suppose that $i \geqslant 1$; then 
\begin{align*}
tP \subseteq (uv)^i t' P \subseteq (uv)^i P \subseteq \langle u \rangle.
\end{align*}
Since $P \nsubseteq \langle u \rangle$ and, as noted previously, $\langle u \rangle$ is a prime ideal of $T$, we find that $t \in \langle u \rangle$; similarly, $t \in \langle v \rangle$. Therefore $t \in \langle u \rangle \cap \langle v \rangle = \langle uv \rangle$, contradicting minimality of $i$. Hence $t' \in T$, and so $\End_T(P) = T$. \\
\indent It follows that every nonzero prime ideal of $P$ of $T$ satisfies $O_\ell(P) = T$, and similarly also satisfies $O_r(P) = T$. Thus $T$ is a maximal order.
\end{proof}

\indent We will use the same approach as in the above proof to show that $T = A \hash G$ is a maximal order in cases (ii) and (iii). As before, it suffices to show that $\End_T(P) = T$ for all height 1 primes, which we show to be true for a few carefully chosen primes. These primes are chosen so that when we invert powers of their generators, the resulting algebra $T'$ is Azumaya. We then show that $Z(T')$ is a UFD, which will allow us to deduce that $T \subseteq \End_T(P) \subseteq T'$ for all remaining primes $P$. Finally, we argue that necessarily $\End_T(P) = T$. \\
\indent It turns out that we must work harder to prove Theorem \ref{maxorderthm} for the remaining two cases, mainly because it is more difficult to show that our carefully selected primes are in fact prime. Moreover, showing that the centres of our Azumaya skew group algebras are UFDs is more involved.

\subsection{Proof of Theorem \ref{maxorderthm} for case (ii)}
We have already seen that the action is $X$-outer in this case, but again Theorem \ref{martinthm} is difficult to apply for the same reasons as for case (i). Recall that, by Proposition \ref{iiiazumaya}, $T'' \coloneqq A'' \hash S_2$ is Azumaya, where $A' = \Bbbk_{-1}[u^{\pm 1},v^{\pm 1}]$ and $A'' = A'[(u^2-v^2)^{-1}]$. \\
\indent We write $A = \Bbbk_{-1}[u,v]$ and $G = S_2 = \langle h \rangle$ throughout this subsection, where $ h \cdot u = v, \gap h \cdot v = u$. We will need the following result:

\begin{lem} \label{primesforc}
Let $T = \Bbbk_{-1}[u,v] \hash S_2$, as in case \emph{(ii)}. Then $\langle uv \rangle$ and $\langle u^2-v^2 \rangle$ are prime ideals of $T$.
\end{lem}
\begin{proof}
\indent We first consider $\langle uv \rangle$. Since this ideal is $G$-stable, it suffices to show that the quotient $T/\langle uv \rangle \cong (\Bbbk_{-1}[u,v]/\langle uv \rangle) \hash S_2$ is prime. Equivalently, by \cite[Lemma 6.17]{gw}, we show that its classical quotient ring 
\begin{align*}
Q\big( (\Bbbk_{-1}[u,v]/\langle uv \rangle \hash S_2 \big) \cong Q\big( \Bbbk_{-1}[u,v]/\langle uv \rangle \big) \hash S_2
\end{align*}
is simple. First observe that $Q\big( \Bbbk_{-1}[u,v]/\langle uv \rangle \big) \cong \Bbbk(u) \times \Bbbk(v)$.
Tracing through this isomorphism, we find that the corresponding $S_2$-action is given by
\begin{align*}
h \cdot (f_1(u),f_2(v)) = (f_2(u),f_1(v)),
\end{align*}
and so we wish to show that $\big(\Bbbk(u) \times \Bbbk(v)\big) \hash S_2$ is simple. By \cite[Theorem 1.2 (c)]{oinert}, it suffices to show that $\Bbbk(u) \times \Bbbk(v)$ is $G$-simple, and that the centre of $\big(\Bbbk(u) \times \Bbbk(v)\big) \hash S_2$ is a field. For the first of these, let $I$ be a nonzero $G$-stable ideal of $\Bbbk(u) \times \Bbbk(v)$, and let $0 \neq (f_1(u),f_2(v)) \in I$, where, acting by $h$ if necessary and using $G$-stability, we may assume that $f_1(u) \neq 0$. Multiplying by $(f_1(u)^{-1},0)$, we find that $(1,0) \in I$, and then acting by $h$ shows that $(0,1) \in I$, so that $(1,1) \in I$ and hence $I = \Bbbk(u) \times \Bbbk(v)$. Therefore, $\Bbbk(u) \times \Bbbk(v)$ is $G$-simple. By Lemma \ref{outercentre}, the centre of $\big(\Bbbk(u) \times \Bbbk(v)\big) \hash S_2$ equals $\big(\Bbbk(u) \times \Bbbk(v)\big)^{S_2}$, which is easily seen to be $\{ (f(u),f(v)) \mid f(t) \in \Bbbk(t) \}$,
and this is clearly a field. Therefore $\big(\Bbbk(u) \times \Bbbk(v)\big) \hash S_2$ is simple, and hence $(\Bbbk_{-1}[u,v]/\langle uv \rangle) \hash S_2$ is prime, which shows that $\langle uv \rangle$ is a prime ideal of $T$. \\
\indent We now show that $\langle u^2 - v^2 \rangle$ is a prime ideal of $T$. As before, it suffices to show that the quotient $(\Bbbk_{-1}[u,v]/\langle u^2-v^2 \rangle) \hash S_2$ is prime, or equivalently, that $Q(\Bbbk_{-1}[u,v]/\langle u^2-v^2 \rangle) \hash S_2$ is simple. We first claim that we have an isomorphism
\begin{align*}
Q\left( \frac{\Bbbk_{-1}[u,v]}{\langle u^2-v^2 \rangle}\right) \hspace{-1pt} \hash S_2 \cong M_2(\Bbbk(t)) \hash S_2
\end{align*}
with an appropriate action of $S_2$ on $M_2(\Bbbk(t))$. To this end, define an algebra homomorphism
\begin{gather*}
\phi : \frac{\Bbbk_{-1}[u^{\pm 1},v^{\pm 1}]}{\langle u^2-v^2 \rangle} \to M_2(\Bbbk[t^{\pm 1}]), \\
\phi(u) = 
\begin{pmatrix}
0 & 1 \\
t & 0
\end{pmatrix}, 
\quad
\phi(v) = i
\begin{pmatrix}
0 & 1 \\
-t & 0
\end{pmatrix},
\end{gather*}
where $i^2 = -1$, which is easily checked to be well-defined. Since $\Bbbk_{-1}[u^{\pm 1}, v^{\pm 1}]$ and $ M_2(\Bbbk[t^{\pm 1}])$ are both free modules of rank 4 over $R \coloneqq \Bbbk[u^{\pm 2}]$ and $\phi(R) = \Bbbk[t^{\pm 1}]$ respectively, we see that $\phi$ is an isomorphism. Chasing through the definition of $\phi$, one can verify that it is an isomorphism of $S_2$-modules provided that we define
\begin{align*}
h \cdot e_{11} = e_{22}, \quad h \cdot e_{12} = -it e_{21}, \quad h \cdot e_{21} = i t^{-1} e_{12}, \quad h \cdot e_{22} = e_{11},  \quad h \cdot t = t.
\end{align*} 
With this action, we therefore have a chain of isomorphisms
\begin{align*}
Q\left( \frac{\Bbbk_{-1}[u,v]}{\langle u^2-v^2 \rangle} \right) \hspace{-2pt} \hash S_2 
\cong Q\left( \frac{\Bbbk_{-1}[u^{\pm 1},v^{\pm 1}]}{\langle u^2-v^2 \rangle} \right) \hspace{-2pt} \hash S_2
\cong Q(M_2(\Bbbk[t^{\pm 1}]) \hash S_2
\cong M_2(\Bbbk(t)) \hash S_2.
\end{align*}
By direct calculation, one can verify that the centre of $M_2(\Bbbk(t)) \hash S_2$ is
\begin{align*}
Z \coloneqq \left \{ \begin{pmatrix} a & 0 \\ 0 & a \end{pmatrix} + \begin{pmatrix} 0 & b \\ -itb & 0 \end{pmatrix} \hspace{-1pt} h \hspace{2pt} \Bigg | \hspace{2pt} a, b \in \Bbbk(t) \right \} .
\end{align*}
This is a field because the map 
\begin{align*}
\psi : \Bbbk(t)[x] \to Z, \quad t \mapsto \begin{pmatrix} t & 0 \\ 0 & t \end{pmatrix}, \quad x \mapsto \begin{pmatrix} 0 & 1 \\ -it & 0 \end{pmatrix} \hspace{-1pt} h
\end{align*}
gives rise to an isomorphism $Z \cong \Bbbk(t)[x]/\langle x^2 + it \rangle$, where the right hand side is a field. Additionally, $M_2(\Bbbk(t))$ is simple and therefore, by \cite[Theorem 1.2 (c)]{oinert}, it follows that  $M_2(\Bbbk(t)) \hash S_2$ is simple. Hence $\langle u^2-v^2 \rangle$ is a prime ideal of $T$, as claimed.
\end{proof}

\begin{thm} \label{maxordercasec}
The algebra $T = \Bbbk_{-1}[u,v] \hash S_2$ in case \emph{(ii)} is a maximal order.
\end{thm}
\begin{proof}
As in the proof of Theorem \ref{bmaxorder}, it suffices show that $O_\ell(P) = T = O_r(P)$ for all height $1$ primes. So suppose that $P$ has height $1$. We first remark that $\langle u \rangle = \langle v \rangle$ is not a prime ideal of $T$, since the quotient of $T$ by this ideal is isomorphic to $\Bbbk S_2$, which is not prime. If $P$ contains $u^2-v^2$, then by Lemma \ref{primesforc} $P = \langle u^2-v^2 \rangle$, and since $u^2-v^2$ is a normal nonzerodivisor, we have $O_\ell(P) = T$ in this case. Similarly, using Lemma \ref{primesforc} again, if $P$ contains $uv$ then $P = \langle uv \rangle$ and $O_\ell(P) = T$. Now let $P$ be a height $1$ prime of $T$ not containing $u^2-v^2$ or $uv$, so that $P$ corresponds to a height 1 prime $PT''$ of $T'' \coloneqq \Bbbk_{-1}[u^{\pm 1}, v^{\pm 1}][(u^2-v^2)^{-1}] \hash S_2$. As established in Proposition \ref{iiiazumaya}, $T''$ is Azumaya, so there exists a height 1 prime $\mathfrak{p}$ of $Z(T'')$ such that $PT'' = \mathfrak{p} T''$. But using the fact that $Z(RX^{-1}) = Z(R)X^{-1}$ if $X \subseteq Z(R)$ and that $(RX^{-1})^G = R^G X^{-1}$ if $X \subseteq R^G$,
\begin{align*}
Z(T'') &= Z \Big(\Bbbk_{-1}[u^{\pm 1}, v^{\pm 1}][(u^2-v^2)^{-2}] \hash S_2 \Big) \\
&= Z \big(\Bbbk_{-1}[u^{\pm 1}, v^{\pm 1}][(u^2-v^2)^{-2}] \big)^{S_2} \\
&= \big( \Bbbk[u^{\pm 2}, v^{\pm 2}][(u^2-v^2)^{-2}] \big)^{S_2} \\
&= \big( \Bbbk[u^2, v^2][(u^2v^2)^{-1}] \big)^{S_2}[(u^2-v^2)^{-2}] \\
&= \Bbbk[(u^2v^2)^{\pm 1}, u^2 + v^2][(u^2-v^2)^{-2}] \\
&\cong \Bbbk[x^{\pm 1}, y][(y^2-4x)^{-1}].
\end{align*}
The last ring is the localisation of a UFD which implies that $Z(T'')$ is a UFD, and so height 1 primes are principal, which means that $\mathfrak{p} = z Z(T'')$ for some $z \in Z(T'')$. Since $z$ is a central nonzerodivisor, we have $\End_{T''}(zT'') = T''$, and therefore we have a chain of inclusions
\begin{align*}
T \subseteq \End_T(P) \subseteq \End_{T''}(PT'') = \End_{T''}(zT'') = T''.
\end{align*}
\indent It remains to show that this forces $\End_T(P) = T$. To this end, let $t'' \in \End_T(P) \subseteq T''$ and choose $i \geqslant 0$ minimal such that $t \coloneqq (uv(u^2-v^2))^i t'' \in T$. We claim that $i=0$, forcing $t'' \in T$. Seeking a contradiction, suppose that $i \geqslant 1$; then 
\begin{align*}
tP \subseteq (uv(u^2-v^2))^i t'' P \subseteq (uv(u^2-v^2))^i P \subseteq \langle uv \rangle.
\end{align*}
Since $P \nsubseteq \langle uv \rangle$ and, by Lemma \ref{primesforc}, $\langle uv \rangle$ is a prime ideal of $T$, we find that $t \in \langle uv \rangle$; similarly, $t \in \langle u^2-v^2 \rangle$. Hence $t \in \langle uv \rangle \cap \langle u^2-v^2 \rangle = \langle uv(u^2-v^2) \rangle$, contradicting minimality of $i$. Therefore $t'' \in T$, and so $\End_T(P) = T$. \\
\indent It follows that every nonzero prime ideal of $P$ of $T$ satisfies $O_\ell(P) = T$, and similarly also satisfies $O_r(P) = T$, and so $T$ is a maximal order.
\end{proof}

\subsection{Proof of Theorem \ref{maxorderthm} for case (iii)}
We finally come to what turns out to be the most involved case. Again, we use the same approach as in the proof of Theorem \ref{bmaxorder}, but we must first make some preliminary calculations. Most notably, computing the centres of the algebras of interest is quite involved, and so we state these as independent lemmas. \\
\indent Throughout, we write $A = \Bbbk_{-1}[u,v]$, $A' = A[u^{-1},v^{-1}]$, $A'' = A'[(u^{2n}-v^{2n})^{-1}]$, and $T, T', T''$ for the corresponding skew group rings coming from the action of $G = D_n$.

\begin{prop} \label{propcentreodd}
Suppose that $n$ is odd. Then
\begin{align*}
Z(T'') &=\Bbbk[u^2v^2, u^{2n} + v^{2n}][(u^2v^2)^{-1}, ((u^{2n}-v^{2n})^2)^{-1} ] \cong \Bbbk[x^{\pm 1},y][(y^2-4 x^n)^{-1}], 
\end{align*}
which is a UFD.
\end{prop}
\begin{proof}
By the proof of Proposition \ref{disazumayaodd}, $T$ satisfies the hypotheses of Lemma \ref{outercentre}. Writing $a = u^2, \gap b = v^2$, we therefore have
\begin{align*}
Z(T) = Z(A)^{D_n} = \Bbbk[a,b]^{D_n}.
\end{align*}
where the $D_n$-action on $\Bbbk[a,b]$ is given by 
\begin{align*}
g \cdot a = \omega^{2i} a, \quad g \cdot b = \omega^{-2i} b, \quad h \cdot a = b, \quad h \cdot b = a.
\end{align*}
By \cite[Appendix A]{benson},
\begin{align*}
\Bbbk[a,b]^{D_n} = \Bbbk[ab, a^n + b^n]
\end{align*}
is a polynomial ring in two variables, and so
\begin{align*}
Z(T) = \Bbbk[u^2v^2, u^{2n} + v^{2n}] \cong \Bbbk[x,y],
\end{align*}
where $x \coloneqq u^2 v^2, y \coloneqq u^{2n}+v^{2n}$. \\
\indent Finally, to determine $Z(T'')$, note that $T'' = T[(u^2v^2)^{-1}, ((u^{2n}-v^{2n})^2)^{-1}]$, where the multiplicative set generated by $u^2v^2$ and $(u^{2n}-v^{2n})^2$ is contained in $Z(T)$. With the notation for $x$ and $y$ as above, we have $(u^{2n}-v^{2n})^2 = y^2 - 4x^n$, and so
\begin{align*}
Z(T'') = Z(T)[(u^2v^2)^{-1}, ((u^{2n}-v^{2n})^2)^{-1}] \cong \Bbbk[x^{\pm 1},y][(y^2-4 x^n)^{-1}].
\end{align*}
This, being the localisation of a UFD, is itself a UFD.
\end{proof}

Determining the centre of $T''$ is more involved when $n$ is even, essentially because the action is not $X$-outer.

\begin{prop} \label{propcentreeven}
Suppose that $n$ is even and write $m = n/2$. Then
\begin{align*}
Z(T'') \cong \frac{\Bbbk[x,y,z]}{\langle x^2y + y^{m+1} + z^2 \rangle}[y^{-1}, (x^4 + x^2 y^m)^{-1}].
\end{align*}
Moreover, $Z(T'')$ is a UFD.
\end{prop}
\begin{proof}
We first determine the centre of $T$. As in the proof of Proposition \ref{disazumayaeven}, we have an isomorphism
\begin{align*}
T \cong (\Bbbk_{-1}[u,v] \hash \langle g^{m} \rangle) * D_{m}.
\end{align*}
Now,
\begin{align*}
Z(\Bbbk_{-1}[u,v] \hash \langle g^m \rangle) \cong \Bbbk[a,b,c]/\langle ab-c^2 \rangle
\end{align*}
where $a \coloneqq u^2, b \coloneqq v^2, c \coloneqq i uvg^m$. As in the proof of Proposition \ref{disazumayaeven}, the set of automorphisms $\{ \alpha_{\sigma^i \tau^j} \mid 0 \leqslant i < m, j=0,1 \}$ acts as a group of $X$-outer automorphisms on $\Bbbk[a,b,c]/\langle ab-c^2 \rangle$ via
\begin{gather*}
\alpha_{\sigma^i} (a) = \varepsilon^i a, \quad \alpha_{\sigma^i} ( b) = \varepsilon^{-i} b, \quad \alpha_{\sigma^i} (c) = c, \\
\alpha_{\sigma^i \tau} (a) = \varepsilon^{-i} b, \quad \alpha_{\sigma^i \tau} (b) = \varepsilon^i a, \quad \alpha_{\sigma^i \tau} (c) = -c,
\end{gather*}
where $\varepsilon = \omega^2$. Therefore, using Lemma \ref{outercentre}, we have
\begin{align*}
Z(\Bbbk_{-1}[u,v] \hash D_n) = Z((\Bbbk_{-1}[u,v] \hash \langle g^m \rangle) * D_m) = Z(\Bbbk_{-1}[u,v] \hash \langle g^m \rangle)^{D_m} \cong \left(\frac{\Bbbk[a,b,c]}{\langle ab-c^2 \rangle} \right)^{D_m}.
\end{align*}
We first determine $R \coloneqq \Bbbk[a,b,c]^{D_m}$, where the action is as above. We use Molien's formula \cite[Theorem 2.5.2]{benson} to work out the Hilbert series of $R$:
\begin{align*}
\hilb R = \frac{1}{| D_m | } \sum_{\alpha \in G} \frac{1}{\det (I - \alpha t)}.
\end{align*}
In matrix form, the elements of $D_m$ and the relevant determinants are as follows:
\vspace{3pt}
\begin{center}
\begin{tabular}{ c | c | c | c} 
Element & Matrix & Number & $\det(I - \alpha t)$ \\
\hline
\vspace{-2pt} & \vspace{-2pt} & \vspace{-2pt} & \vspace{-2pt}\\
$\sigma^i, \quad 0 \leqslant i \leqslant m-1$ & $\begin{pmatrix} \varepsilon^i & 0 & 0 \\ 0 & \varepsilon^{-i} & 0 \\ 0 & 0 & 1 \end{pmatrix}$ & $m$ & $(1-t)(1-\varepsilon^i t)(1- \varepsilon^{-i}t)$ \\ 
\vspace{-2pt} & \vspace{-2pt} & \vspace{-2pt} & \vspace{-2pt}\\
$\sigma^i \tau, \quad 0 \leqslant i \leqslant m-1$ & $\begin{pmatrix} 0 & \varepsilon^i & 0 \\ \varepsilon^{-i} & 0 & 0 \\ 0 & 0 & -1 \end{pmatrix}$ & $m$ & $(1-t)(1+t)^2$ \\ 
\end{tabular}
\end{center}
\vspace{3pt}
We therefore have
\begin{align*}
\hilb R &= \frac{1}{2m} \left( \frac{m}{(1-t)(1+t)^2} + \sum_{i=0}^{m-1} \frac{1}{(1-t)(1-\varepsilon^i t)(1- \varepsilon^{-i}t)} \right) \\
&= \frac{1}{2(1-t)(1+t)^2} + \frac{1}{2(1-t)} \cdot \underbrace{ \frac{1}{m}\sum_{i=0}^{m-1} \frac{1}{(1-\varepsilon^i t)(1- \varepsilon^{-i}t)}}_{(*)}.
\end{align*}
Now $(*)$ is, by Molien's formula, the Hilbert series of the coordinate ring of an $\mathbb{A}_{m-1}$ singularity, and so is known to equal $(1-t^{2m})/(1-t^2)(1-t^m)^2$. A routine calculation then shows that
\begin{align*}
\hilb R &=  \frac{1-t^{2(m+1)}}{(1-t^2)^2(1-t^m)(1-t^{m+1})}.
\end{align*}
This implies that $R$ has four generators, of degrees $2,2,m,m+1$, and that there is a single relation of degree $2(m+1)$ between these generators. It is easy to check that
\begin{align*}
x \coloneqq \frac{i}{2}(a^m+b^m), \quad y \coloneqq ab, \quad y' \coloneqq c^2, \quad z \coloneqq \frac{1}{2}(a^m c - b^m c),
\end{align*}
are $D_m$-invariants and that
\begin{align*}
x^2y' + y^m y' + z^2 = - \frac{1}{4}(a^{2m} + 2 a^m b^m + b^{2m})c^2 + a^m b^m c^2 + \frac{1}{4} (a^{2m}c^2 - 2 a^m b^m c^2 + b^{2m} c^2) = 0
\end{align*}
so that
\begin{align*}
R = \Bbbk[a,b,c]^{D_m} \cong \frac{\Bbbk[x,y,y',z]}{\langle x^2y' + y^m y' + z^2 \rangle}.
\end{align*}
Since $y = y'$ in $\Bbbk[a,b,c]/\langle ab - c^2 \rangle$, it immediately follows that
\begin{align*}
\left( \frac{\Bbbk[a,b,c]}{\langle ab-c^2 \rangle} \right)^{D_m} \cong \frac{\Bbbk[x,y,z]}{\langle x^2y + y^{m+1} + z^2 \rangle},
\end{align*}
where $x,y,$ and $z$ are as above. Finally, recalling that $Z(\Bbbk_{-1}[u,v] \hash D_n) \cong (\Bbbk[a,b,c]/\langle ab-c^2 \rangle)^{D_m}$ where $a \coloneqq u^2, b \coloneqq v^2, c \coloneqq i uvg^m$, we find that 
\begin{align*}
Z(\Bbbk_{-1}[u,v] \hash D_n) \cong \frac{\Bbbk[x,y,z]}{\langle x^2y + y^{m+1} + z^2 \rangle},
\end{align*}
where we have set
\begin{align*}
x \coloneqq \frac{i}{2}(u^{n}+v^{n}), \quad y \coloneqq u^2v^2, \quad z \coloneqq \frac{i}{2}(u^{n+1} v g^m - u v^{n+1} g^m).
\end{align*}
\indent To determine the centre of $T''$, observe that we have
\begin{align*}
T'' = T[(u^2v^2)^{-1}, ((u^{2n}-v^{2n})^2)^{-1}],
\end{align*}
where the multiplicative set generated by $u^2v^2$ and $(u^{2n}-v^{2n})^2$ is contained in $Z(T)$. Observe also that, with the notation for $x,y,z$ as above, we have
\begin{align*}
\tfrac{1}{16} (u^{2n}-v^{2n})^2 = x^4 + x^2y^m,
\end{align*}
and so
\begin{align*}
Z(T'') = Z(T)[y^{-1}, (x^4 + x^2y^m)^{-1}] \cong \frac{\Bbbk[x,y,z]}{\langle x^2y + y^{m+1} + z^2 \rangle}[y^{-1}, (x^4 + x^2 y^m)^{-1}].
\end{align*}
\indent For the final claim, observe that $S \coloneqq Z(T'')$ is a localisation of the coordinate ring $R$ of a $\mathbb{D}_{m+2}$ singularity. Write $\text{Cl}\gap(R)$ for the divisor class group of $R$. By Nagata's theorem, the canonical map
\begin{align*}
\phi : \text{Cl}\gap(R) \to \text{Cl}\gap(S), \quad \sum_i \alpha_i R \gap \gap \mapsto \gap \gap \sum_i \alpha_i S
\end{align*}
is a surjection. Therefore, if $\phi$ maps every element of $\text{Cl}\gap(R)$ to $[S]$ in $\text{Cl}\gap(S)$ then $\text{Cl}\gap(S)$ is trivial. Since $R$ is a noetherian integrally closed domain \cite[Proposition 1.1.1]{benson}, we can define $\text{Cl}\gap(R)$ to be the set of isomorphism class of rank one reflexive modules, with multiplication given by $[I][J] = [(I \otimes_R J)^{**}]$. Since $\idim R = 2$, these are precisely the rank one maximal Cohen-Macaulay modules. These modules are known: by \cite[9.21]{leu}, up to isomorphism they are
\begin{align*}
\left \{
\begin{array}{l l l}
R, \hspace{4pt} \langle y,z \rangle, \hspace{4pt} \langle z, xy-iy^{m/2 + 1} \rangle, \hspace{4pt} \langle z, xy+iy^{m/2 + 1} \rangle  & & \text{if } m \text{ is even} \\
R, \hspace{4pt} \langle y,z \rangle, \hspace{4pt} \langle x, z+iy^{(m+1)/2} \rangle, \hspace{4pt} \langle xy, z+iy^{(m+1)/2} \rangle  & & \text{if } m \text{ is odd}
\end{array}
\right. .
\end{align*}
But $x,y,$ and $z$ are all invertible in $S$ since
\begin{gather*}
x \cdot (x^3 + xy^m)(x^4 + x^2y^m)^{-1} = 1, \\
y \cdot y^{-1} = 1, \\
z \cdot -x^2 z y^{-1} (x^4 + x^2 y^m)^{-1} = (x^4y + x^2y^{m+1})y^{-1} (x^4 + x^2 y^m)^{-1} = 1.
\end{gather*}
Therefore each of the above ideals gets sent to $[S]$ under $\phi$, and so $\text{Cl}\gap(S)$ is trivial, which implies that $S$ is a UFD.
\end{proof}


The remainder of the proof of Theorem \ref{maxorderthm} for case (iii) does not depend on the parity of $n$. We now show that certain ideals of $T$ are prime:

\begin{lem} \label{primesford}
Let $T = \Bbbk_{-1}[u,v] \hash D_n$, as in case \emph{(iii)}. Then $\langle uv \rangle$ and $\langle u^{2n} - v^{2n} \rangle$ are prime ideals of $T$.
\end{lem}
\begin{proof}
Since $\langle uv \rangle$ is $G$-stable, we need to show that $T/\langle uv \rangle \cong \Bbbk_{-1}[u,v]/\langle uv \rangle \hash D_n$ is prime. Equivalently, we show that $Q(\Bbbk_{-1}[u,v]/\langle uv \rangle) \hash D_n$ is simple, where as in the proof of Proposition \ref{primesforc}, we have
\begin{align*}
Q(\Bbbk_{-1}[u,v]/\langle uv \rangle) \hash D_n \cong (\Bbbk(u) \times \Bbbk(v)) \hash D_n,
\end{align*}
where $D_n$ acts via
\begin{align*}
g \cdot (u,0) = (\omega u, 0), \quad g \cdot (0,v) = (0, \omega^{-1} v), \quad h \cdot (u,0) = (0,v), \quad h \cdot (0,v) = (u,0).
\end{align*}
To show that this latter ring is simple, we show that $\Bbbk(u) \times \Bbbk(v)$ is $G$-simple and that the centraliser $C$ of $\Bbbk(u) \times \Bbbk(v)$ in $((\Bbbk(u) \times \Bbbk(v)) \hash D_n$ is $\Bbbk(u) \times \Bbbk(v)$, see \cite[Theorem 6.13]{oinert2}. $G$-simplicity of $\Bbbk(u) \times \Bbbk(v)$ follows from the same argument as in the proof of Lemma \ref{primesforc}, so now let $c = \sum_{0 \leqslant i < n, \gap 0 \leqslant j \leqslant 1} f_{ij} g^ih^j \in C$, where $f_{ij} = (f_{ij}^{(1)}(u),f_{ij}^{(2)}(v)) \in \Bbbk(u) \times \Bbbk(v)$. Then
\begin{align*}
\sum_{\genfrac{}{}{0pt}{}{0 \leqslant i < n}{0 \leqslant j \leqslant 1}} (uf_{ij}^{(1)},0) \gap g^ih^j = (u,0) \gap c = c \gap (u,0) = \sum_{0 \leqslant i < n} (\omega^i u f_{i0}^{(1)},0) \gap g^i + \sum_{0 \leqslant i < n} (0,\omega^{-i} v f_{i1}^{(2)}) g^i h,
\end{align*}
which implies that $f_{i0}^{(1)} = 0$ for all $1 \leqslant i < n$ and that $f_{i1}^{(1)} = 0 = f_{i1}^{(2)}$ for all $0 \leqslant i < n$. A similar calculation with $(0,v)$ in place of $(u,0)$ shows that $f_{i0}^{(2)} = 0$ for all $1 \leqslant i < n$, and so $c = (f_{00}^{(1)}, f_{00}^{(1)})$. It follows that $C = \Bbbk(u) \times \Bbbk(v)$, and so $(\Bbbk(u) \times \Bbbk(v)) \hash D_n$ is simple. Therefore $Q(\Bbbk_{-1}[u,v]/\langle uv \rangle) \hash D_n$ is simple, and so $\langle uv \rangle$ is a prime ideal of $T$. \\
\indent We now consider $\langle u^{2n} - v^{2n} \rangle$. Since $\langle u^2 - v^2 \rangle$ is a prime ideal of $\Bbbk_{-1}[u,v]$ \cite[Section 8]{irving1}, and $\langle u^{2n} - v^{2n} \rangle = \bigcap_{f \in D_n} f \cdot \langle u^2-v^2 \rangle$, it follows that $\langle u^{2n} - v^{2n} \rangle$ is a $G$-prime ideal of $A$, and so we wish to show that $T/\langle u^{2n} - v^{2n} \rangle \cong (A/ \langle u^{2n} - v^{2n} \rangle) \hash D_n$ is prime. Now, $T/\langle u^{2n} - v^{2n} \rangle$ is a $G$-prime ring and $\langle u^2-v^2 \rangle / \langle u^{2n}-v^{2n} \rangle$ is a minimal prime of $T/\langle u^{2n} - v^{2n} \rangle$ with stabiliser $H = \{1,h\}$, and so by \cite[Corollary 14.8]{passman}, it suffices to show that $ (\Bbbk_{-1}[u,v]/\langle u^2-v^2 \rangle) \hash H$ is prime (where here we recall that $H$ acts via $h \cdot u = v, \gap h \cdot v = u$). But this is established in Lemma \ref{primesforc}, and so $\langle u^{2n} - v^{2n} \rangle$ is a prime ideal of $T$.
\end{proof}

This allows us to prove that the remaining case is a maximal order, which completes the proof of Theorem \ref{maxorderthm}.

\begin{thm} \label{maxordercased}
The algebra $T = \Bbbk_{-1}[u,v] \hash D_n$ in case \emph{(iii)} is a maximal order.
\end{thm}
\begin{proof}
Again, we only show that $O_\ell(P) = T$ for all height $1$ primes which, as with Theorem \ref{maxordercaseb}, will be enough to prove the result. So suppose that $P$ is a height $1$ prime. We first remark that $\langle u \rangle = \langle v \rangle$ is not a prime ideal of $T$, since the quotient of $T$ by this ideal is isomorphic to $\Bbbk D_n$, which is not prime. If $P$ contains $u^{2n}-v^{2n}$, then by Lemma \ref{primesford} $P = \langle u^{2n}-v^{2n} \rangle$, and since $u^{2n}-v^{2n}$ is a normal nonzerodivisor, we have $O_\ell(P) = T$ in this case. Similarly, using Lemma \ref{primesford} again, if $P$ contains $uv$ then $P = \langle uv \rangle$ and $O_\ell(P) = T$. Now let $P$ be a height $1$ prime of $T$ not containing $u^{2n}-v^{2n}$ or $uv$, so that $P$ corresponds to a height 1 prime $PT''$ of $T'' \coloneqq \Bbbk_{-1}[u^{\pm 1}, v^{\pm 1}][(u^{2n}-v^{2n})^{-1}] \hash D_n$. As established in Propostions \ref{disazumayaodd} and \ref{disazumayaeven}, $T''$ is Azumaya, and so there exists a height 1 prime $\mathfrak{p}$ of $Z(T'')$ such that $PT'' = \mathfrak{p} T''$. But $Z(T'')$ is a UFD by Propositions \ref{propcentreodd} and \ref{propcentreeven}, so height 1 primes are principal, and so $\mathfrak{p} = z Z(T'')$ for some $z \in Z(T'')$. Since $z$ is a central nonzerodivisor, we have $\End_{T''}(zT'') = T''$, and therefore we have a chain of inclusions
\begin{align*}
T \subseteq \End_T(P) \subseteq \End_{T''}(PT'') = \End_{T''}(zT'') = T''.
\end{align*}
It remains to show that this forces $\End_T(P) = T$. To this end, let $t'' \in \End_T(P) \subseteq T''$ and choose $i \geqslant 0$ minimal such that $t \coloneqq (uv(u^{2n}-v^{2n}))^i t'' \in T$. We claim that $i=0$, forcing $t'' \in T$. Seeking a contradiction, suppose that $i \geqslant 1$; then 
\begin{align*}
tP \subseteq (uv(u^{2n}-v^{2n}))^i t'' P \subseteq (uv(u^{2n}-v^{2n}))^i P \subseteq \langle uv \rangle.
\end{align*}
Since $P \nsubseteq \langle uv \rangle$ and, by Lemma \ref{primesforc}, $\langle uv \rangle$ is a prime ideal of $T$, we find that $t \in \langle uv \rangle$; similarly, $t \in \langle u^{2n}-v^{2n} \rangle$. Therefore $t \in \langle uv \rangle \cap \langle u^{2n}-v^{2n} \rangle = \langle uv(u^{2n}-v^{2n}) \rangle$, contradicting minimality of $i$. Hence $t'' \in T$, and so $\End_T(P) = T$. \\
\indent It follows that every nonzero prime ideal of $P$ of $T$ satisfies $O_\ell(P) = T$, and similarly also satisfies $O_r(P) = T$, and so $T$ is a maximal order.
\end{proof}

\section{A proof of Auslander's Theorem for quantum Kleinian singularities}
Using the fact that, for a quantum Kleinian singularity $A^G$, the algebra $A \hash G$ is a maximal order, one can show that Auslander's Theorem holds for these algebras. We use the following quite general result:

\begin{lem} \label{endoringlem}
Let $R$ be a prime Goldie maximal order and let $e \in R$ an idempotent. Then $\End_{eRe}(Re) \cong R$. 
\end{lem}
\begin{proof}
Since $R$ is prime, we have an embedding of $R$ into $\End_{eRe}(Re)$, where each element of $R$ gives rise to an endomorphism via left multiplication. Let $Q$ be the Goldie quotient ring of $R$. Then,  by \cite[Theorem 3]{small}, $eQe$ is the Goldie quotient ring of $eRe$. Since $Q$ is simple, we have $QeQ = Q$, and we claim that in fact $ReQ = Q$. Indeed, since $QeQ = Q$, there exist $r_i, s_i \in R$ and regular elements $x_i, y_i \in R$ such that $\sum_{i=1}^n x_i^{-1} r_i e y_i^{-1} s_i = 1$, and we may as well assume we have a common denominator and write $\sum_{i=1}^n x^{-1} r_i e y_i^{-1} s_i = 1$ for some regular element $x \in R$. Then, pre- and post-multiplying by $x$ and $x^{-1}$, respectively,
\begin{align*}
1 = \sum_{i=1}^n r_i e y_i^{-1} s_i x^{-1} \in ReQ,
\end{align*}
so that $ReQ = Q$. Hence 
\begin{align}
Qe = ReQe \cong Re \otimes_{eRe} eQe \cong Re \otimes_{eRe} Q(eRe). \label{injhullRe}
\end{align}
Since $Qe$ is torsionfree as a right $eQe$-module, $Qe$ and $Re$ are both torsionfree over $eRe$. From (\ref{injhullRe}) and \cite[Exercise 7F]{gw}, we deduce that $Qe$ is the $eRe$-injective hull of $Re$. Now consider any $f \in \End_{eRe}(Re)$, and view it as a map $Re \to Qe$. Since $Qe$ is an injective right $eRe$-module, $f$ lifts to a morphism $f' : Qe \to Qe$, and we claim that this lifting is unique. So suppose that $f', f''$ are two such lifts, and let $qe \in Qe$. By essentially the same proof as that of \cite[Lemma 4]{small}, there exists a regular element $ere \in eRe$ such that $qere = qe \cdot ere \in Re$. Then
\begin{align*}
(f'(qe)-f''(qe)) ere = f'(qere)-f''(qere) = 0.
\end{align*}
Since $Qe$ is $eRe$-torsionfree, this implies that $f'(qe) = f''(qe)$, and so $f' = f''$. Thus, given any $f \in \End_{eRe}(Re)$, we can lift to a unique $f' \in \End_{eRe}(Qe)$, so we get an embedding $\End_{eRe}(Re) \subseteq \End_{eRe}(Qe)$. We also have $\End_{eRe}(Qe) = \End_{eQe}(Qe)$ by \cite[Exercise 10H (b)]{gw}, where this last endomorphism ring is isomorphic to $Q$ since $Q$ and $eQe$ are Morita equivalent. Therefore we have inclusions $R \subseteq \End_{eRe}(Re) \subseteq Q$. But if $q \in \End_{eRe}(Re)$ then $qe \in Re \subseteq R$ so that $ \End_{eRe}(Re) \cdot e \subseteq R$. Therefore $R$ and $\End_{eRe}(Re)$ are equivalent orders in $Q$, and since $R$ is a maximal order, we must have $\End_{eRe}(Re) = R$.
\end{proof}

We now use the above lemma to show that Auslander's Theorem holds for the algebras $A \hash G$. This result was established in \cite[Theorem 4.1]{ckwz} but using very different techniques. We remark that our proof can be adapted to show that Auslander's Theorem holds for PBW deformations of the algebras $A \hash G$, which can be defined as in \cite{walton}. A proof of this result does not appear to be possible using the techniques of \cite{ckwz}.

\begin{thm} \label{auslandercor}
Suppose that $A^G$ is a quantum Kleinian singularity. Then
\begin{align*}
\End_{A^G}(A) \cong A \hash G.
\end{align*}
\end{thm}
\begin{proof}
The ring $R = A \hash G$ is a prime noetherian maximal order by Lemma \ref{AGprime} and Theorem \ref{maxorderthm}. Moreover, $e = \tfrac{1}{|G|}\sum_{g \in G} g$ is an idempotent with $eRe \cong A^G$ and $Re \cong A$ as an $(A \hash G, A^G)$-bimodule by \cite[Lemma 3.1]{bhz}. The result then follows from Lemma \ref{endoringlem}.
\end{proof}


\bibliographystyle{amsalpha}

\providecommand{\bysame}{\leavevmode\hbox to3em{\hrulefill}\thinspace}
\providecommand{\MR}{\relax\ifhmode\unskip\space\fi MR }
\providecommand{\MRhref}[2]{%
  \href{http://www.ams.org/mathscinet-getitem?mr=#1}{#2}
}
\providecommand{\href}[2]{#2}

\end{document}